\documentclass[11pt]{article}

\usepackage{lastpage,extramarks} 
\usepackage{amsmath,amsfonts,amsthm,thmtools,mathtools,amssymb,mathrsfs,graphicx,graphbox,wrapfig,enumitem} 
\usepackage[a4paper,margin=3cm]{geometry}
\usepackage{setspace}
\usepackage{float}
\usepackage[english]{babel}
\usepackage[x11names]{xcolor}
\usepackage{tikz}
\usetikzlibrary{angles,quotes}
\usepackage[colorlinks=true,urlcolor=.,linkcolor=.,citecolor=.]{hyperref}
\usepackage{imakeidx,hyperref}
\usepackage{pdfpages}
\usepackage[shortcuts]{extdash}
\usepackage{pgfplots}
\usepackage{subfigure}

\definecolor{myblue}{HTML}{4245A3}
\definecolor{mypastel}{HTML}{1C7AA2}
\definecolor{myorange}{HTML}{FF8726}
\definecolor{myred}{HTML}{D75341}

\theoremstyle{definition}
\newtheorem{defn}{Definition}[section]
\newtheorem{lem}[defn]{Lemma}
\newtheorem{prop}[defn]{Proposition}

\newtheorem{conject}[defn]{Conjecture}
\newtheorem*{conject*}{Conjecture}
\newtheorem*{16piconjecture}{$ 16\pi $-Conjecture}

\theoremstyle{plain}
\newtheorem{theorem}[defn]{Theorem}
\newtheorem*{theorem*}{Theorem}
\newtheorem{thmCite}{Theorem}

\theoremstyle{remark}

\newtheorem{rem}[defn]{Remark}
\newtheorem{openq}[defn]{Question}

\newcommand{\com}[1]{}

\newcommand{\N}{\mathbb{N}}
\newcommand{\Z}{\mathbb{Z}}

\newcommand{\R}{\mathbb{R}}

\renewcommand{\S}{\mathbb{S}}
\newcommand{\eps}{\varepsilon}
\newcommand{\abs}[1]{\left\lvert#1\right\rvert}
\newcommand{\norm}[1]{\left\lVert#1\right\rVert}

\DeclareMathOperator\Imm{Imm}
\DeclareMathOperator\CatSph{CatSph}
\DeclareMathOperator\Ann{Ann}
\DeclareMathOperator{\arccosh}{arccosh}

\title{The Willmore Energy Landscape of Spheres 
	and\\ Avoidable Singularities of the Willmore Flow}

\author{Elena M\"ader-Baumdicker\footnote{elena.maeder-baumdicker@fu-berlin.de,\newline ORCID: 0000-0001-8125-8799,\newline Freie Universit\"at Berlin, Germany,\newline Technical University Darmstadt, Germany}, Jona Seidel\footnote{seidel@mathematik.tu-darmstadt.de,\newline ORCID: 0009-0000-1282-8209,\newline Technical University Darmstadt, Germany}}
\date{}

\begin{document}

\maketitle

\begin{abstract}
	
	We study the sublevel sets of the Willmore energy on the space of smoothly immersed $ 2 $-spheres in Euclidean $ 3 $-space. We show that the subset of immersions with energy at most $ 12\pi $ consists of four regular homotopy classes.
	Moreover, we show that, in certain regular homotopy classes, all singularities of the Willmore flow are avoidable---that is, the initial surface admits a regular homotopy to a round sphere whose Willmore energy does not exceed that of the initial surface. This yields a classification of initial surfaces with energy at most $ 12\pi $ that lead to unavoidable singularities.
	As a further consequence, we obtain an extension of the Li--Yau inequality at $ 12\pi $ for a large class of immersed spheres without triple points.
	To prove these results, we glue together different instances of the Willmore flow and employ an invariant for triple-point-free immersed spheres.\footnote{Mathematics Subject Classification \href{https://mathscinet.ams.org/mathscinet/search/mscdoc.html?code=53C42,53E40,57R42}{53C42; 53E40, 57R42},\newline Keywords: 
		Regular Homotopy Class, Willmore Flow, Singularity, Surgery
	}

\end{abstract}

\newpage
\tableofcontents

\section{Introduction and Main Results}
\subsection{The Willmore Energy}
For a closed surface $ \Sigma $ and an immersion $ f:\Sigma\to\R^3 $ into Euclidean space, the Willmore energy is given by
\begin{equation*}
	W(f) := \int_\Sigma H^2 d\mu_f,
\end{equation*}
where $ \mu_f $ is the area measure induced by $ f $ and $ H=\frac12 \left(\kappa_1 + \kappa_2\right) $ is the mean curvature of $ f $.
In this paper, we concentrate on immersed spheres, that is, $ \Sigma=\S^2 $.

The Willmore energy is of considerable analytic interest: its associated Euler–Lagrange equations are a system of quasi-linear, fourth-order partial differential equations. The search for minimizers of the Willmore energy \cite{LY,PS_WillmoreSurfs,Simon_W,BK_ExistenceWillmore,R_AnalysisWillmore,BR_Quant,MN_WillmoreConjecture}, as well as the study of the corresponding $ L^2 $-gradient flow---the Willmore flow \cite{KS_Gradient,KS_SmallInitial,Simonett_WillmoreFlow,KS_Remov,LN,LS_Wtori,DMSS_WillmoreFlowTOR,DMRS_WillmoreFlowToriFixedConformal,PR_Param}---has been the focus of extensive recent study.
From a geometric perspective, the Willmore energy is remarkable in being invariant under scaling and conformal transformations; in fact, up to a topological constant, it is the unique global invariant with this full conformal invariance \cite{MN_Conformal}.
It also serves as a bending energy and a measure of complexity as it penalizes self-intersections and deviations from round spheres, making it relevant in applications across biology, physics and computer graphics \cite{Helfrich,M_Germain,Koerber_HawkingMass,BKPAL_Mesh}. We refer the reader to \cite{KS_Survey,R_WillmoreNotes,MN_WillmoreConjectureSurvey} for an overview of the Willmore energy and its applications.

In the following section, we outline the approach taken in this paper and relate it to known results. The main contribution of this paper is stated in Section~\ref{sec_mainResults} as Theorem~\ref{thm_globalSing}, with further results and applications in Sections~\ref{sec_avoidableSingularities}, \ref{sec_extensionLiYau}, \ref{sec_optimalSE} and~\ref{sec_gen16pi}. The remaining sections are mostly dedicated to the proof of Theorem~\ref{thm_globalSing}.

\subsection{The Willmore Energy Landscape}
An energy landscape is the graph of an energy function over a set of admissible objects, in this case the Willmore functional over the set $ \Imm(\S^2,\R^3) $ of smoothly immersed $ 2 $-spheres in Euclidean $ 3 $-space. 
Usually, a landscape is described by its critical points and their indices. The Willmore energy landscape of spheres has been thoroughly studied from this Morse-theoretic viewpoint.
All its immersed critical points, the \textit{Willmore spheres}, are classified \cite{Bryant_Dual}. They are inverted complete minimal surfaces with finite total curvature and embedded planar ends and appear at quantized energy levels $ 4\pi k $ for some $ k\in\N $. The absolute minima at $ 4\pi $ are the \textit{round spheres}, that is, reparametrizations, scalings and translations of the standard embedding $ e:\S^2\to\R^3 $. After $ 4\pi $, the next Willmore spheres appear at $ 16\pi $. They form a four-parameter family up to reparametrization and conformal transformation \cite{Bryant_Dual} and their indices are equal to one \cite{HMB,HKMB}.
Beyond smoothly immersed Willmore spheres, branched Willmore spheres have been the subject of recent research \cite{KL_W22conformal,CL_BranchedWillmore,LN,M_BWSIndex,MR_BWS_S3S4,Martino_BWS} (with differing notions of ``branched'', see \cite{Martino_BWS}).
These results all provide \textit{local} information about the energy landscape.

The main object of this paper is to obtain insight into the \textit{global} topology of the Willmore energy landscape. 
\com{About the underlying space of immersed spheres it is known that it contains one regular homotopy class \cite{Smale_SE} and the first fundamental group modulo rotation is the integers \cite{MB_SE}. We want to describe the topology of the Willmore energy landscape.}
One approach is to cut off the landscape at certain energy levels and analyze the structure of the resulting ``valleys'', that is, the regular homotopy classes of the sublevel sets 
\begin{equation*}
	I(\omega):=\left\{ f\in \Imm(\S^2,\R^3) \ \mid \ W(f) \leq \omega  \right\},
\end{equation*}
for $ \omega\geq 4\pi $. A purely topological result by Smale shows that the whole space $ I(\infty)=\Imm(\S^2,\R^3) $ has one regular homotopy class \cite{Smale_SE}. But to tackle the question for finite energy values, one has to combine tools from topology and geometric analysis.

For $ \omega\leq 8\pi $, such tools are available and they imply
\begin{equation}\label{eq_I8pihas2RHC}
	I(\omega) \text{ has two regular homotopy classes for } \omega\in[4\pi,8\pi].
\end{equation} 
To prove this, we first employ the Willmore flow to get an upper bound on the number of regular homotopy classes of $ I(\omega) $. Kuwert and Schätzle proved that, up to $ 8\pi $, the Willmore flow behaves well.
\begin{thmCite}{(\cite{KS_Remov}).}\label{thm_KS8pi}
	Let $ f_0:\S^2\to\R^3 $ be a smooth immersion of a sphere with Willmore energy $$ W(f_0)\leq 8\pi. $$ Then, the Willmore flow with initial data $ f_0 $ exists smoothly for all times and converges to a round sphere.
\end{thmCite}
Recall that all orientation-preserving diffeomorphisms of $ \S^2 $ are homotopic to each other \cite{Smale_DiffS2} and thus Theorem~\ref{thm_KS8pi} implies that $ I(8\pi) $ has at most two regular homotopy classes, since every $ f\in I(8\pi) $ can be joined to a round sphere of either orientation by the Willmore flow.

Second, a lower bound on the number of regular homotopy classes of $ I(\omega) $ arises from a topological study of self-intersections carried out by Max and Banchoff \cite{MB_SE} (with alternative proofs \cite{Hughes_SE,Gor_Local}). We say that a regular homotopy $ f:\S^2\times[0,1]\to\R^3 $ has an \textit{$ n $-tuple point} if there exist $ t\in [0,1] $ and $ x\in \R^3 $ such that $ \left(f(\cdot,t)\right)^{-1}(\{x\}) $ contains at least $ n $ points. 
\begin{thmCite}{(\cite{MB_SE}).}\label{thm_MBquad}
	Every regular homotopy between two round spheres of opposite orientation has a quadruple point.
\end{thmCite}
The decisive link between the topology of self-intersections and the Willmore energy is the Li--Yau inequality.
\begin{thmCite}{(\cite{LY}).}\label{thm_LYineq}
	If a smooth immersion $ f:\Sigma\to\R^k $ of a closed surface has an $ n $-tuple point, then $$ W(f)\geq 4\pi n. $$
\end{thmCite}
As a consequence, every regular homotopy between two round spheres of opposite orientation must pass through a time at which the immersion has Willmore energy at least $16\pi$. Thus, round spheres of opposite orientation cannot lie in the same regular homotopy class in $ I(\omega) $ for $ \omega<16\pi $. Consequently, for $ \omega\in[4\pi,16\pi) $, the sublevel set $ I(\omega) $ has at least two regular homotopy classes. With Theorem~\ref{thm_KS8pi}, we arrive at \eqref{eq_I8pihas2RHC}.

The aim of this paper is to extend this approach to $ \omega\in(8\pi,12\pi] $, where Theorem~\ref{thm_KS8pi} is not available because singularities of the Willmore flow might prevent the convergence to a round sphere. 

\subsection{Main Results}\label{sec_mainResults}
For any immersion $ f:\S^2\to\R^3 $, we denote by $ -f $ the immersion defined by $ (-f)(x):=-f(x) $.
\begin{defn}\label{def_J}
	Consider a smooth immersion $ j:\S^2\to\R^3 $ obtained by rotating the curve displayed in Figure~\ref{fig_sketchJ}. Let $ J $ be the set of all smooth immersions $ f:\S^2\to\R^3 $ that admit a regular homotopy without triple points to $ j $ or $ -j $.
\end{defn}
\begin{figure}[H]
	\centering
	\vspace*{-0.05\textwidth}
	\includegraphics[align=c,width=0.49\textwidth]{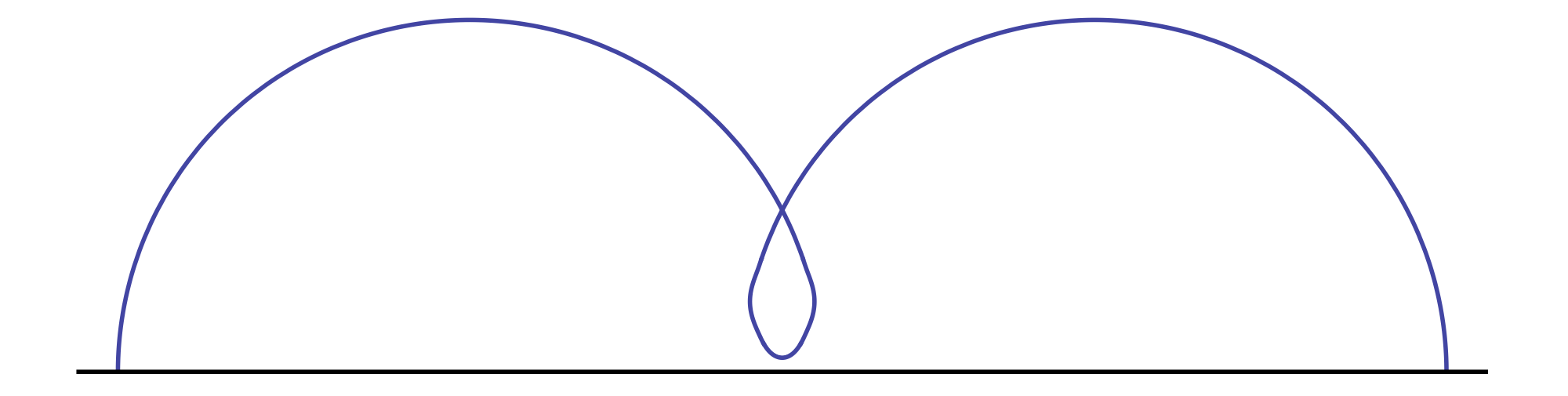}
	\includegraphics[align=c,width=0.49\textwidth]{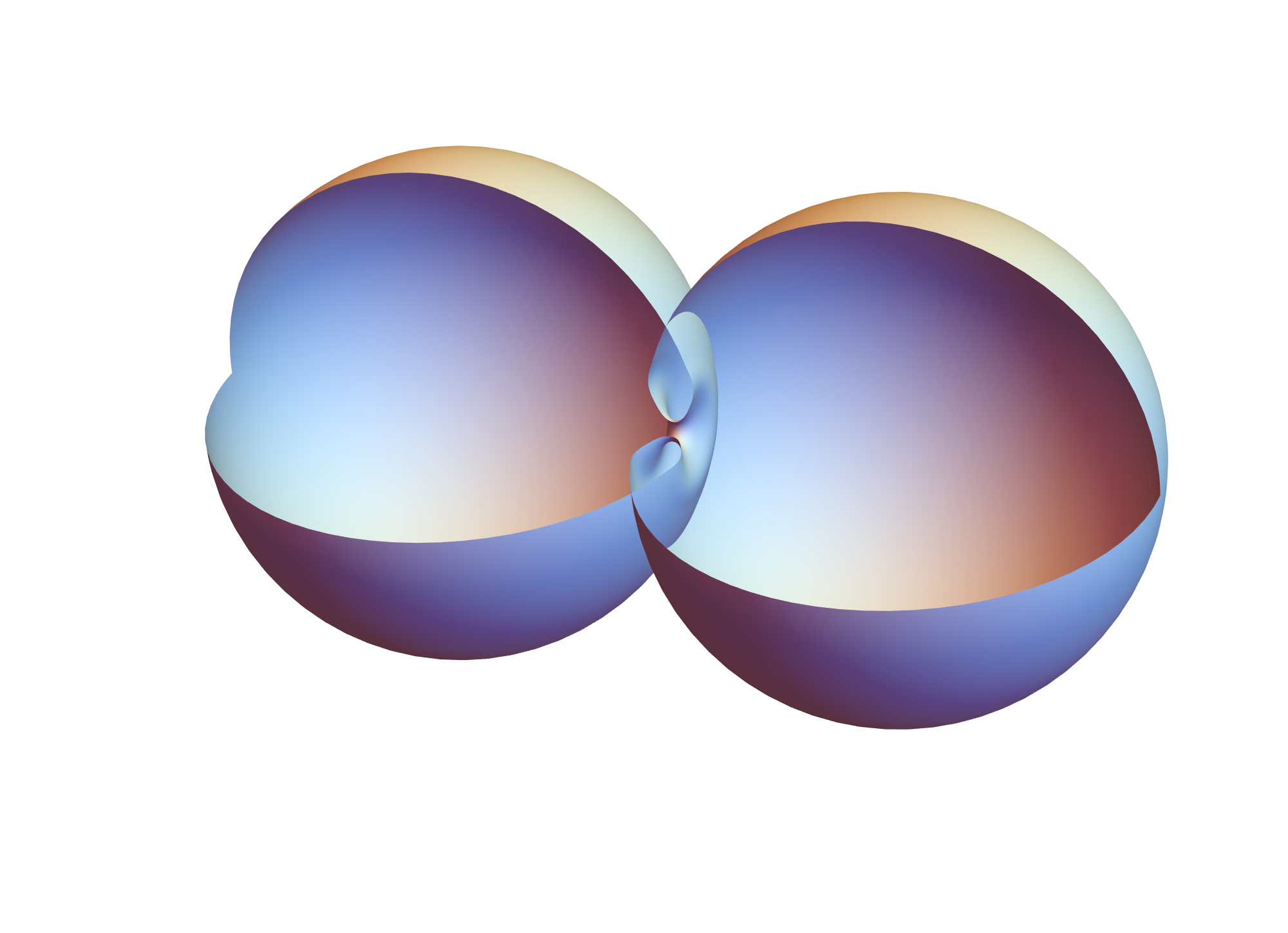}
	\vspace*{-0.05\textwidth}
	\caption{An immersion in the family $ J $ and its generating curve.}\label{fig_sketchJ}
\end{figure}
In \cite{MS_Num,Blatt}, surfaces of revolution of this kind serve as initial surfaces that lead to singularities. However, surfaces in $ J $ are not necessarily rotationally symmetric and may possess additional self-intersections. In fact, the definition of $ J $ is rather implicit, but note that for any given smooth immersion $ f:\S^2\to\R^3 $, an invariant devised in \cite{S_TriplePoints} gives an explicit way of checking whether $ f $ belongs to $ J $.

We show that $ J $ defines model surfaces that can be reached by regular homotopies of low Willmore energy whenever the Willmore flow fails to converge to a round sphere. The proof relies crucially on the classification of blow-ups at the singular time by Lamm and Nguyen \cite{LN}. The main contribution of this paper is the following (see Section~\ref{sec_surgery} for the proof). For homotopies $ H:\S^2\times I \to\R^3 $, we denote $ H_t:=H(\cdot,t) $.
\begin{restatable}{theorem}{globalSing}\label{thm_globalSing}
	If a smooth immersion $ f:\S^2\to\R^3 $ has Willmore energy $ W(f)\in(4\pi,12\pi] $, then there exists a regular homotopy $ H: \S^2 \times [0,1]\to\R^3 $ which satisfies $ H_0=f $, $ W(H_t)<W(f) $ for all $ t\in(0,1] $ and $ H_1 $ is either a round sphere or equal to some $ j\in J $.
\end{restatable}
The immersion $ j $ in Theorem~\ref{thm_globalSing} can be taken to have Willmore energy arbitrarily close to $ 8\pi $.

An invariant introduced by the second author in \cite{S_TriplePoints} shows that the model surfaces $ J $ cannot be omitted in Theorem~\ref{thm_globalSing}. It relies on the structure of self-intersecting immersed spheres, similar to that revealed in Theorem~\ref{thm_MBquad}, though the invariant is obtained using entirely different methods.
\begin{thmCite}{(\cite[Theorem~1.7]{S_TriplePoints}).}\label{thm_triplePoints}
	Every regular homotopy joining some $ j\in J $ to an embedded sphere has triple points.
\end{thmCite}
The Li--Yau inequality immediately gives $ W(H_t)\geq 12\pi $ for some time $ t $. This can be improved to a strict inequality (see Appendix~\ref{sec_appendixGlue}).
\begin{restatable}{cor}{corTriplePoints}\label{cor_triplePoints}
	For every regular homotopy $ H:\S^2\times [0,1] \to \R^3$ from some $ j\in J $ to an embedded sphere, there exists a time $ t\in[0,1] $ with $ W(H_t)>12\pi $.
\end{restatable}
The two theorems combined yield the following.
\begin{theorem}\label{thm_RHClasses}
	$ I(\omega) $ has four regular homotopy classes for $ \omega\in(8\pi,12\pi] $.
\end{theorem}
\begin{figure}[H]
	\centering
	\resizebox{0.9\textwidth}{!}{
	\begin{tikzpicture}[scale=1.3]
	
		\def\a{1.7}
		\def\x{1/\a}
		\def\d{0.3}
		
		\foreach \y/\label in {1/4\pi, 2/8\pi, 3/12\pi} {
			\draw[dotted, thin, black] (-7.5*\x-\d,\y) -- (8.5*\x+\d,\y);
			\node[anchor=south east] at (8.5*\x+\d,\y) {\(\label\)};
		}
		
		\fill[gray, opacity=0.2] 
		(-7*\x-\d,3)
		-- (-6*\x-\d,2)
		-- (-5*\x-\d,3)
		-- cycle;
		
		\fill[gray, opacity=0.2] 
		(-5*\x,3)
		-- (-3*\x,1)
		-- (-\x,3)
		-- cycle;
		
		\fill[gray, opacity=0.2] 
		(\x,3)
		-- (3*\x,1)
		-- (5*\x,3)
		-- cycle;
		
		\fill[gray, opacity=0.2] 
		(5*\x+\d,3)
		-- (6*\x+\d,2)
		-- (7*\x+\d,3)
		-- cycle;
		
		\draw[thick, black]
		(-7*\x-\d,3)
		-- (-6*\x-\d,2)
		-- (-5*\x-\d,3);
		\draw[thick, black]
		(-5*\x,3)
		-- (-3*\x,1)
		-- (-\x,3);
		\draw[thick, black]
		(\x,3)
		-- (3*\x,1)
		-- (5*\x,3);
		\draw[thick, black]
		(5*\x+\d,3)
		-- (6*\x+\d,2)
		-- (7*\x+\d,3);
		
		\node[anchor=center] at (-6*\x-\d-0.02,2.8) {\includegraphics[width=1.2cm]{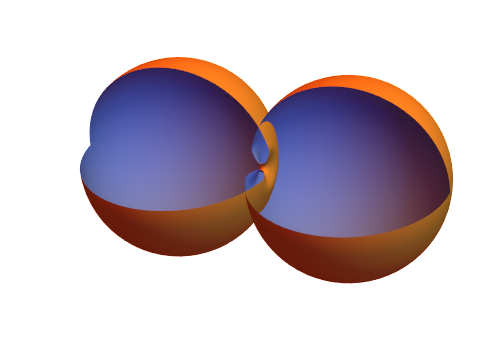}};
		\node[anchor=center] at (-3*\x-0.02,1-\d) {\includegraphics[width=0.7cm]{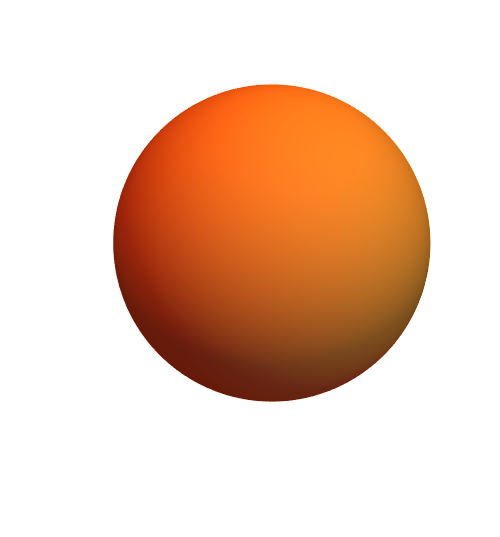}};
		\node[anchor=center] at (3*\x-0.02,1-\d) {\includegraphics[width=0.7cm]{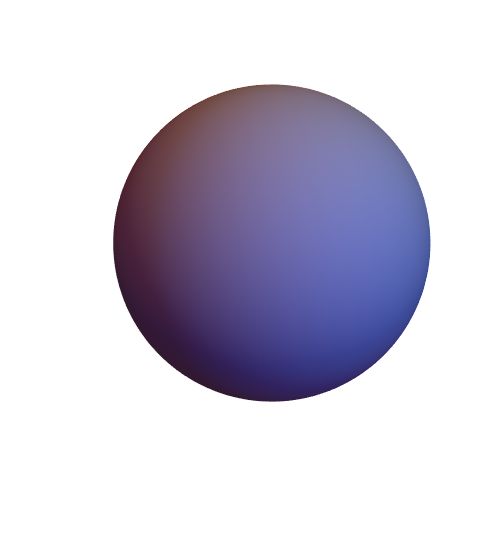}};
		\node[anchor=center] at (6*\x+\d-0.02,2.8) {\includegraphics[width=1.2cm]{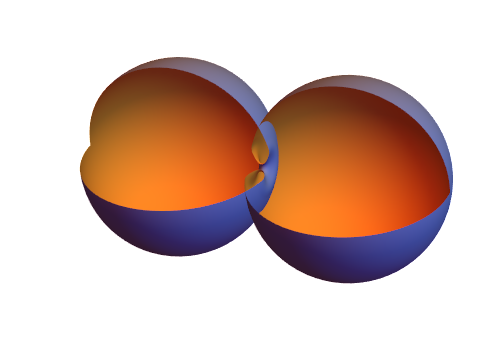}};
		
	\end{tikzpicture}
	}
	\caption{A sketch of the Willmore energy landscape of immersed $ 2 $-spheres in $ \R^3 $ below $ 12\pi $.}\label{fig_landscape}
\end{figure}
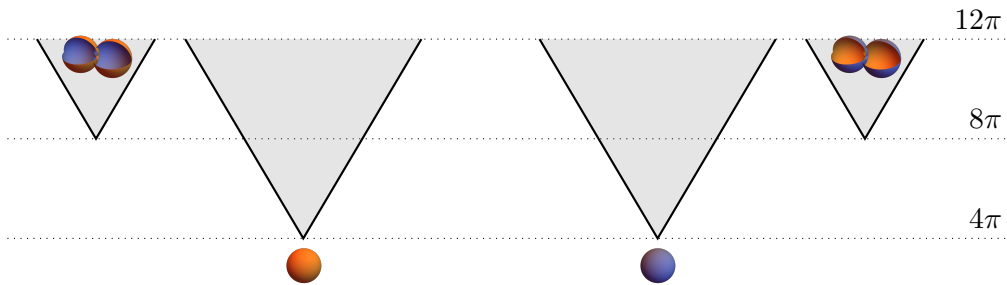
\begin{proof}[Proof of Theorem~\ref{thm_RHClasses}]
	As a consequence of Theorem~\ref{thm_globalSing}, there exists $ j\in J \, \cap \, I(\omega) $ such that all $ f\in I(\omega) $ can be joined to $ j $, $ -j $, $ e $ or $ -e $ by a regular homotopy in $ I(\omega) $ (see Lemma~\ref{lem_Jlem}).
	Therefore, $ I(\omega) $ has at most four regular homotopy classes.
	In $ I(\omega) $, the immersion $ j $ is not regularly homotopic to $ -j $ since all regular homotopies from any $ f\in\Imm(\S^2,\R^3) $ to $ -f $ have a quadruple point \cite{MB_SE}.
	By Corollary~\ref{cor_triplePoints}, it is also not regularly homotopic to $ e $ or $ -e $ in $ I(\omega) $. This proves the statement.
\end{proof}
In the two regular homotopy classes $ J_+ $ and $ J_- $ that do not contain round spheres, the infimum $$  \inf_{j\in J_{\pm}} W(j) = 8\pi $$ is not attained (see Remark~\ref{rem_JW>8pi}).

\subsection{Avoidable Singularities}\label{sec_avoidableSingularities}
For every initial surface, there exists a Willmore flow $ f:\S^2\times [0,T)\to\R^3 $ on a maximal interval $ [0,T) $ with  $ 0<T\leq \infty $ \cite{KS_Survey}. If $ T<\infty $, then $ f $ cannot be smoothly extended onto $ T $.
\begin{defn}
	We say that a maximal Willmore flow $ f:\S^2\times [0,T)\to\R^3 $ \textit{develops a singularity} if it does not converge smoothly to a Willmore sphere as $ t\to T $. This singularity is \textit{avoidable} if there exists a regular homotopy $ H:\S^2 \times [0,1]\to\R^3 $ from $ f_0 $ to a smooth Willmore sphere with $ W(H_t)<W(f_0) $ for all $ t\in(0,1] $. If no such homotopy exists, we call the singularity \textit{unavoidable}.
\end{defn}
Recall that below $ 16\pi $ the only smooth Willmore spheres are the round spheres \cite{Bryant_Dual}. Whether the Willmore flow can develop singularities in finite time, and whether the area remains bounded, are both open questions. No such assumptions are imposed in the definition.

Blatt showed in \cite{Blatt} that the Willmore flow starting in certain surfaces of revolution in $ J $ necessarily develops singularities. If it were to converge to a round sphere, this would imply the existence of a regular homotopy between plane curves of different turning numbers. This example was originally proposed in \cite{MS_Num}, where numerical simulations suggested singularity formation in finite time. In \cite{DMSS_WillmoreFlowTOR}, a family of singular examples for tori of revolution was given following an analogous argument. Apart from surfaces of revolution, no other initial surfaces are currently known to produce singularities under the Willmore flow.
As a consequence of Theorem~\ref{thm_triplePoints}, not only do surfaces of revolution in $ J $ develop singularities under the Willmore flow, but so does any other initial surface in $ J $.
We obtain a full classification of initial surfaces with energy at most $ 12\pi $ that lead to unavoidable singularities.
\begin{prop}\label{prop_unavoidableSingularities}
	Let $ f_0:\S^2\to\R^3 $ be a smooth immersion with $ W(f_0) \leq 12\pi $. Then, the Willmore flow with initial surface $ f_0 $ develops an unavoidable singularity if and only if $ f_0\in J $.
\end{prop}
\begin{proof}
	If $ f_0\in J $, then, by Theorem~\ref{thm_triplePoints}, its Willmore flow necessarily develops an unavoidable singularity; otherwise, it would provide a regular homotopy to a round sphere without triple points. If, conversely, the Willmore flow with initial surface $ f_0 $ develops an unavoidable singularity, then, by Theorem~\ref{thm_globalSing}, there exists a regular homotopy $ H $ with $ W(H_t)<W(f_0) $ for $ t\in(0,1] $, ending in some $ j\in J$. In particular, $ H $ has no triple points, and it follows that $ f_0\in J $.
\end{proof}
These results do not describe the analytic behavior of the Willmore flow. In particular, it remains open whether the Willmore flow can develop avoidable singularities. However, it seems reasonable to expect that singularity formation is governed not only by the global topology of self-intersections but also by local analytic behavior.
\begin{conject}
	There exists a smooth immersion $ f_0:\S^2\to\R^3 $ with $ W(f_0)\leq 12\pi $ such that the Willmore flow with initial surface $ f_0 $ develops an avoidable singularity.
\end{conject}
One might find an initial surface $ f_0\notin J $ that locally is very close to a catenoid singularity (e.g., three of the four cases in Figure~\ref{fig_catenoidSpheres}). One would need to show that such a (global) singularity model indeed appears as a singularity of the flow. Similar arguments have been made for the Ricci flow \cite{Stolarski_Box} and the mean curvature flow \cite{LeeZhao_Box}. However, these flows develop singularities in finite time, whereas it remains an open question whether the Willmore flow does as well.

\begin{rem}
	If one is interested in flowing ``through'' singularities and allowing for topology change, then the techniques used in Section~\ref{sec_globalSing} provide results for such a notion of ``Willmore flow with surgery'' in the very specific setting of spheres with energy at most $ 12\pi $ (see Remark~\ref{rem_surgery}). 
\end{rem}

\subsection{An Extension of the Li--Yau Inequality at $ 12\pi $}\label{sec_extensionLiYau}
The Li–Yau inequality provides a rare and powerful link between the global topology of an immersed surface—particularly its self-intersections---and its Willmore energy, an integral of a fundamentally local quantity. It has also been extended to related settings, such as for curves \cite{MR_LY1dim,Miura_LiYau} and for surfaces with boundary \cite{Schlierf_LiYauBoundary}. 

The Li--Yau inequality supports the idea that Willmore energy serves as a measure of complexity: more ``complicated'' immersions tend to have higher Willmore energy. An immersion with a triple point, for instance, has Willmore energy at least $ 12\pi $. Similarly, an immersion with arbitrarily short closed geodesic has Willmore energy at least $ 6\pi $ \cite{MRS_closedGeodesic}. 
Taking genus into account, the Willmore Conjecture states that every immersed torus has Willmore energy at least $ 2\pi^2 $. The Li--Yau inequality reduces the problem to embedded tori, and the conjecture was famously resolved by Marques and Neves \cite{MN_WillmoreConjecture}. Moreover, the minimal energy for fixed genus $ g $ tends to $ 8\pi $ as $ g\to\infty $ \cite{KLS_LargeGenus}.

However, beyond the relations to genus and the Li–Yau inequality, such direct connections between the topological complexity of a closed surface and its Willmore energy remain rare.
We present two results in this spirit. The first is a lower bound of $ 12\pi $ for a large family of surfaces without triple points; the second is a lower bound for spheres of revolution in terms of the turning number of the generating curve. Both results crucially depend on the Li--Yau inequality.

Theorems~\ref{thm_globalSing} and~\ref{thm_triplePoints} provide an explicit description of the self-intersections of immersed spheres with Willmore energy at most $ 12\pi $. This yields the following extension of the Li--Yau inequality at $ 12\pi $. We denote by
\begin{align*}
	F:\left\{ f:\S^2\to\R^3 \text{ smooth immersion without triple points } \right\} \to \Z^\Z
\end{align*}
the invariant defined in \cite{S_TriplePoints} and by $ e_k:=(\delta_{ik})_{i\in\Z} $ the standard basis vectors of $ \Z^\Z $. 
\begin{prop}\label{prop_LiYauExtension}
	Let $ f:\S^2\to\R^3 $ be a smooth immersion without triple points. If $ F(f)\notin \{e_{-3},e_{-1},e_1,e_3\} $, then $ W(f)> 12\pi $.
\end{prop}
\begin{proof}
	Suppose $ W(f)\leq 12\pi $. In \cite[Example~5.2]{S_TriplePoints}, it is computed that $ F(\pm e)=e_{\pm 1} $ and $ F(j)=e_{\pm 3} $ for all $ j\in J $.
	Then, Theorem~\ref{thm_globalSing} provides a regular homotopy without triple points to round spheres $ e $, $ -e $ or some $ j\in J $. The invariant is constant along regular homotopies without triple points \cite[Theorem~1.3]{S_TriplePoints} and thus $ F(f)\in \{e_{-3},e_{-1},e_1,e_3\}. $
\end{proof}
The assumptions of Proposition~\ref{prop_LiYauExtension} are satisfied by a large family of immersions. See \cite[Section~6]{S_TriplePoints} for a description of the image of $ F $ and Figure~\ref{fig_examplesExtendedLiYau} for some examples. Note that the proposition applies not only to the displayed surfaces but also to any deformation without triple points that may break any symmetry.
\begin{figure}[H]
	\centering
	\begin{tikzpicture}
		\node[anchor=south] at (0,0) {\includegraphics[width=0.95\textwidth]{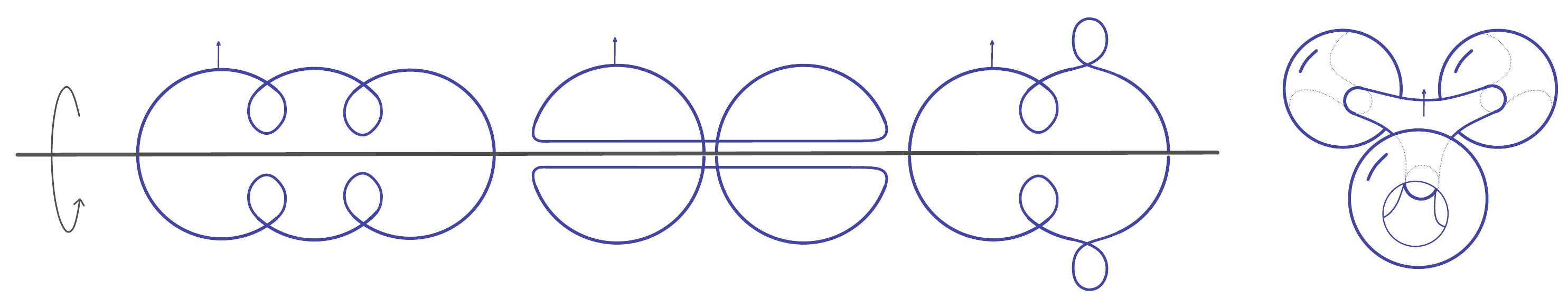}};
		\def\dx{0.25\textwidth}
		
		\foreach \k/\label in {-1.15/$F=2e_3-e_1$, -0.2/$F=2e_1 - e_{-1}$, 0.65/$F=e_3+e_{-1}-e_{1}$, 1.55/$F=3e_{-1}-2e_1$} {
			\node[anchor=south] at (\k*\dx,-0.5) {\small \label};
		}
	\end{tikzpicture}
	\caption{Some immersions without triple points and their values of $ F $ as computed in \cite[Figure~17]{S_TriplePoints}.}\label{fig_examplesExtendedLiYau}
\end{figure}
We mention one further extension of the Li--Yau inequality that is independent of the other results in this paper. It applies only to spheres of revolution. A sphere of revolution is a smoothly immersed sphere that can be parametrized as $ f(t,\varphi)= \left(r(t)\cos\varphi,r(t)\sin\varphi,h(t)\right)$ for some regular curve $ c=(r,h):[0,1]\to[0,\infty)\times \R $ where $ r(0)=r(1)=0 $, $ r(t)>0 $ for all $ t\in(0,1) $ and $ c $ meets the symmetry axis perpendicularly.

\begin{restatable}{prop}{propLiYauTurningNumber}\label{prop_LiYauExtensionRotSymm}
	Let $ f:\S^2\to\R^3 $ be a sphere of revolution whose generating curve $ c:[0,1]\to[0,\infty)\times \R $ has turning number $ \tau\in\Z+\frac12 $. If $ f $ is not a round sphere, then
	\begin{equation*}
		W(f) > 4\pi \left(\abs{\tau}+\frac12\right).
	\end{equation*}
\end{restatable}
A proof can be found in Appendix~\ref{sec_appendixGlue}.
This inequality is sharp as can be seen from gluing round spheres together by catenoid necks (see Remark~\ref{rem_JW>8pi} and \cite{Blatt} for $ \abs{\tau}=\frac32 $). 
For tori of revolution, analogous methods yield the (nonsharp) inequality $ W(f)\geq 4\pi \abs{\tau} $. It is interesting to compare these lower bounds to the energy of Willmore surfaces of revolution (see \cite[eq.~(46)]{Mandel_WillmoreSOR}).

\subsection{An Optimal Sphere Eversion}\label{sec_optimalSE}
We outline the relations of the main results of this paper to sphere eversions and the $ 16\pi $-Conjecture.
A sphere eversion is a regular homotopy connecting two round spheres of opposite orientation. The existence of a sphere eversion was proved by Smale \cite{Smale_SE}, and later explicit sphere eversions were found \cite{Kuiper_SE,Phillips_SE,Morin_SERet,Apery_SE,FSH_SE,Aitchison_Holiverse,HY_SE,BB_SE}, the first by Shapiro \cite{FM_SEShapiro}\com{further see https://chrishills.org.uk/ChrisHills/sphereeversion/}. Theorem~\ref{thm_MBquad} by Max and Banchoff shows that for every sphere eversion $ H $, there exists $ t\in[0,1] $ such that $ H_t $ has a quadruple point. By the Li--Yau inequality, it follows that $ W(H_t)\geq 16\pi $. Hence, every sphere eversion must reach a Willmore energy of at least $ 16\pi $. Does there exist an optimal sphere eversion in the following sense?
\begin{16piconjecture}[Kusner, 1984]
	There exists a sphere eversion $ H:\S^2\times[0,1]\to\R^3 $ that satisfies for all $ t\in[0,1] $ $$ W(H_t)\leq 16\pi. $$
\end{16piconjecture}
Kusner proposed to construct such an optimal sphere eversion by a min-max procedure or by Willmore flow. We summarize the idea of the latter here while the former is treated in \cite{R_SE}.
One explicit sphere eversion is the Froissart-Morin sphere eversion \cite{Morin_SERet}, which uses Morin's surface \cite{Morin_SEProb} as a halfway model. It has the special symmetry that, after a rotation of $ \frac\pi2 $ in $ \R^3 $, the same surface is obtained but of opposite orientation. Given a regular homotopy from such a halfway model to a round sphere, one obtains a sphere eversion whose second half is constructed by reversing the orientation and the temporal direction of the first half.
Kusner found an explicit parametrization of a complete minimal surface whose compactified inversion is the Morin surface and a critical point of the Willmore energy at $ 16\pi $ \cite{Kusner_Conformal}. This surface can be perturbed along its index-one direction \cite{HMB} to reduce its energy below $ 16\pi $. This perturbed surface is not a critical point and thus the Willmore flow can deform it to reduce the Willmore energy. However, it is not known whether the Willmore flow exists for all time and converges to a round sphere. In \cite{FSKBHC_SE}, numerical analysis suggests that it does, but analytically, the possibility of singularities developing under the Willmore flow in this case has not been ruled out. 

We obtain the following.
\begin{prop}\label{prop_RHdecW}
	If a smooth immersion $ f:\S^2\to\R^3 $ has Willmore energy at most $ 12\pi $ and there exists a regular homotopy from $ f $ to a round sphere without triple points, then there also exists one whose Willmore energy does not exceed $ W(f) $.
\end{prop}
\begin{proof}
	The assumption implies $ f\notin J $ by Theorem~\ref{thm_triplePoints} and thus Theorem~\ref{thm_globalSing} yields the statement.
\end{proof}
In other words, for a given initial surface, the topological description of a regular homotopy to a round sphere is sufficient to obtain a regular homotopy of low Willmore energy. 

\begin{openq}\label{q_GeneralizationOfCorollary}
	Does Proposition~\ref{prop_RHdecW} remain true if ``$ 12\pi $'' is replaced by ``$ 16\pi $'' and ``triple points'' is replaced by ``quadruple points''?
\end{openq}
An affirmative answer to the preceding question would prove the $ 16\pi $-Conjecture since the Froissart-Morin sphere eversion provides a regular homotopy from the perturbed Kusner-Morin surface to a round sphere without quadruple points.
This would require a generalization of the main results to $ 16\pi $, which is discussed in the following section. 

\subsection{Generalizations to $ 16\pi $}\label{sec_gen16pi}
A complete description of the Willmore energy landscape would include the number of regular homotopy classes of the sublevel sets $ I(\omega) $, not only for $ \omega\in[4\pi,12\pi] $, but for all $ \omega\in[4\pi,\infty) $. We outline the challenges for the generalization to $ \omega\leq 16\pi $.
\begin{openq}
	What are the regular homotopy classes of $ I(\omega) $ for $ \omega\in(12\pi,16\pi] $?
\end{openq}
Answering this requires a generalization of both Theorem~\ref{thm_globalSing} and Theorem~\ref{thm_triplePoints}.
The former can build on the local classification of \cite{LN} up to $ 16\pi $, but challenges arise. See Section~\ref{sec_globalSing16pi}, which contains a partial result.
For the latter, instead of an invariant for triple-point-free immersions, one would need to devise an invariant for quadruple-point-free immersions that separates round spheres from any $ j\in J $ as well as any other global singularity model that may arise.
\begin{openq}{(\cite[Question~1.8]{S_TriplePoints}).}\label{q_quadruplePoints}
	Does every regular homotopy joining some $ j\in J $ to an embedded sphere have at least one quadruple point?
\end{openq}

Similar results for $ \omega>16\pi $ are out of reach for the approach introduced in this paper. This is due to the fact that results like Theorem~\ref{thm_MBquad} and Theorem~\ref{thm_triplePoints} can only exist for $ n $-tuple points with $ n\leq 4 $: Any regular homotopy of surfaces immersed in $ \R^3 $ can be slightly perturbed to not have any $ n $-tuple points for $ n\geq 5 $ (cf.\ \cite{S_TriplePoints}).

\section{Avoiding Singularities of the Willmore Flow}\label{sec_globalSing}
In this section, we give a global description of singularities of the Willmore flow below $ 12\pi $ using the local classification of \cite{LN}, and prove Theorem~\ref{thm_globalSing}.

We first present the idea of the proof. The details are covered in Sections~\ref{sec_gluingCatSph} and~\ref{sec_surgery}. 
Let $ f:\S^2 \times [0,T)\to \R^3 $ be a maximal Willmore flow with $ W(f_0)\leq 12\pi $ that does not converge to a round sphere. Then, by \cite{LN}, there is a point in $ \R^3 $ around which the image of $ f $ locally converges to a catenoid. This catenoid may be increasing or decreasing in scale. By scale invariance of the Willmore energy, a homotopy that scales the immersion by some factor $ \lambda_t $ with $ \lambda_0=1 $ has constant Willmore energy. In the following, we implicitly assume such rescaling whenever needed. At some time $ s<T $, the immersion $ f_{s} $ is on some annulus $ C\subseteq\S^2 $ sufficiently close to a catenoid \cite{LN}. We replace the catenoid by two flat disks and glue them to $ f|_{D_1} $ and $ f|_{D_2} $ on the remaining disk components $ D_1 $ and $ D_2 $ of $ \S^2\setminus C $. We obtain two immersed spheres, each of which has Willmore energy less than $ 8\pi $. This allows us to start the Willmore flow, which converges to a round sphere \cite{KS_Remov}. These two flows are then glued to the two boundaries of the catenoid $ f|_C $ at each time. To control the Willmore energy of the gluing, we scale up the two flows as needed.
This yields a regular homotopy ending in a small catenoid glued to two round spheres (see Proposition~\ref{prop_gluingA}). Each sphere can be oriented in two ways relative to the catenoid (see Figure~\ref{fig_catenoidSpheres}).
\begin{figure}[H]
	\centering
	\newcommand{\drawCatSph}[3]{
		\def\sphOne{#2}
		\def\sphTwo{#3}
		\begin{tikzpicture}[scale=0.8]
			\begin{axis}[
				scale=1,
				axis equal,
				axis lines = none,
				tick=none,
				samples=500
				]
				\pgfmathsetmacro{\lam}{#1}
				\pgfmathsetmacro{\R}{#1}
				\pgfmathsetmacro{\Ra}{\R+0.75}
				\pgfmathsetmacro{\tzero}{ln(\lam + sqrt(\lam^2 - 1))/\lam}
				\pgfmathsetmacro{\theta}{rad(asin(1/\R))}
				\pgfmathsetmacro{\thetaa}{rad(asin(1/\Ra))}
				
				\addplot[myred, thick, domain=-\tzero:\tzero,parametric] ({x},{cosh(x*\lam)/\lam});
				\addplot[myred, thick, domain=-\tzero:\tzero,parametric] ({x},{-cosh(x*\lam)/\lam});
				
				\if \sphOne1
				\addplot[myblue, thick, domain=\theta:2*pi-\theta,parametric] ({\R*(cos(deg(\theta))-cos(deg(x)))+\tzero},{\R*sin(deg(x))});
				\else
				\addplot[myblue, thick, domain=\thetaa:2*pi-\thetaa,parametric] ({\Ra*(-cos(deg(\thetaa))+cos(deg(x)))+\tzero},{\Ra*sin(deg(x))});
				\fi
				
				\if \sphTwo1
				\addplot[myblue, thick, domain=\theta:2*pi-\theta,parametric] ({-(\R*(cos(deg(\theta))-cos(deg(x)))+\tzero)},{\R*sin(deg(x))});
				\else
				\addplot[myblue, thick, domain=\thetaa:2*pi-\thetaa,parametric] ({-(\Ra*(-cos(deg(\thetaa))+cos(deg(x)))+\tzero)},{\Ra*sin(deg(x))});
				\fi
			\end{axis}
		\end{tikzpicture}
	}
	\resizebox{0.6\textwidth}{!}{
	\begin{tabular}{c c}
		\drawCatSph{6}{1}{1}&
		\drawCatSph{6}{0}{1}\\
		\drawCatSph{6}{1}{0}&
		\drawCatSph{6}{0}{0}
	\end{tabular}
	}
	\caption{Four ways to attach two round spheres to a catenoid.}\label{fig_catenoidSpheres}
\end{figure}
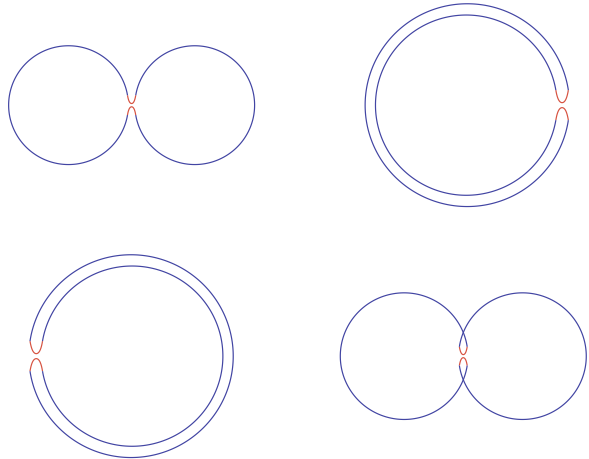
In one of the four cases, we obtain an immersion of the family $ J $, and we are done. For the other cases, the catenoid lies outside of at least one of the spheres, and we can decrease the Willmore energy by shrinking this sphere relative to the catenoid (see Lemma~\ref{lem_catenoidSphereShrinking}). Over time, we see less of both the sphere and the catenoid and eventually drop below a Willmore energy of $ 8\pi $. We then restart the Willmore flow, which converges to a round sphere.
\subsection{Notation}\label{sec_notation}
We denote open balls of radius $ r>0 $ and center $ p\in\R^n $ by $ B_r(p) $.
For a smooth map $ f:\Omega\subseteq \R^m\to\R $, we write
\begin{align*}
	\norm{f}_{C^k(\Omega)}&:=\sum_{\abs{\alpha}\leq k} \sup_{x\in \Omega} \abs{D^\alpha f(x)},\\
	\norm{D^kf}_{C^0(\Omega)}&:=\sum_{\abs{\alpha}= k} \sup_{x\in \Omega} \abs{D^\alpha f(x)},
\end{align*}
and for maps $ f:\Omega\subseteq \S^2\to\R^3 $, we define $ C^k $ norms by fixing charts of $ \S^2 $.
For an interval of radii $ I\subseteq (0,\infty) $, we denote an annulus by $$ \Ann I:=\left\{(x,y)\in\R^2 \ \mid \  \sqrt{x^2+y^2}\in I \right\}. $$
\subsection{Gluing Round Spheres to Catenoids}\label{sec_gluingCatSph}
For the remainder of this section, we fix for all $ \delta>0 $ a gluing function $\phi\in C^\infty(\R,\R)$ with
\begin{align*}
	\phi(x) = \begin{cases}
		0 & \text{if } x\leq1-\delta,\\
		1 & \text{if } x\geq1+\delta,
	\end{cases}
\end{align*}
and the properties $ \abs{\phi}\leq 1, \abs{\phi'}\leq M\delta^{-1}, \abs{\phi''}\leq M\delta^{-2} $ for some constant $ M>0 $ independent of $ \delta $.
A standard mollifier serves to construct such a function.
\com{One example can be defined as follows.
	\begin{align*}
		c&:=\int_{-1}^{1} e^{\frac{1}{x^2-1}} dx,\\
		f(x)&:= \begin{cases}
			\frac1c e^{\frac{1}{x^2-1}} \quad &, -1<x<1,\\
			0 \quad &, \abs{x}\geq 1,
		\end{cases}\\
		g(x)&:=\int_{-1}^{x} f(s)ds.
	\end{align*}
	The function $ g $ is smooth and satisfies $ \abs{g'}< 1 $ and $ \abs{g''}< 1 $. 
	Now, we can set $ \phi(x):=g\left( \frac{x-t}{\delta} \right) $ which has the stated properties.
}
\begin{defn}\label{def_gluing}
	For $ \delta\in(0,1) $, the gluing function $ \phi:\R \to \R $, the annulus\linebreak $ A:=\Ann[1-\delta,1+\delta] $ and two graphs $ f_1,f_2:A\to\R^3  $ with $ f_i(x,y) = (x,y,u_i(x,y)) $ for some $ u_1,u_2\in C^2(A,\R) $, we define the \textit{$ \delta $-gluing} of $ f_1 $ and $ f_2 $ as the graph $ \tilde{f}(x,y)=(x,y,\tilde{u}(x,y)) $ where $ \tilde{u} $ is defined as
	\begin{align*}
		\tilde{u}(x,y):=\left( 1- \phi\left(\sqrt{x^2+y^2}\right) \right) u_1(x,y) + \phi\left(\sqrt{x^2+y^2}\right) u_2(x,y).
	\end{align*}
\end{defn}
\begin{lem}\label{lem_gluingLemma}
	Let $ \delta\in (0,1) $. Then, there exists some constant $ C(M,\delta)>0 $ such that the following holds. For all $A,f_1,f_2,u_1,u_2 $ as above that in addition satisfy $ \norm{u_i}_{C^2(A)}\leq 1 $, the $ \delta $-gluing $ \tilde{f} $ of $ f_1 $ and $ f_2 $ satisfies
	\begin{align*}
		W(\tilde f) \leq W(f_1)+ C \norm{f_2-f_1}_{C^2(A,\R^3)}.
	\end{align*}
\end{lem}
\begin{proof}
	Consider the involved graphs in polar coordinates $ (x,y)= r(\cos\varphi,\sin \varphi) $. Then, the $ \delta $-gluing $ \tilde{u}(r,\varphi):=u_1(r,\varphi) +  \phi(r)\left( u_2(r,\varphi) - u_1(r,\varphi)  \right) $ has the derivatives
	\begin{align*}
		\partial_r \tilde u &= \partial_r u_1 + \phi \cdot \partial_r(u_2-u_1) + \phi' \cdot (u_2-u_1),\\
		\partial_\varphi \tilde u &= \partial_\varphi u_1 + \phi \cdot \partial_\varphi(u_2-u_1),\\
		\partial_{rr} \tilde u &= \partial_{rr}u_1 + \phi \cdot \partial_{rr}(u_2-u_1) + 2\phi' \cdot \partial_r(u_2-u_1) + \phi'' \cdot (u_2-u_1),\\
		\partial_{\varphi\varphi} \tilde u &= \partial_{\varphi\varphi} u_1 + \phi \cdot \partial_{\varphi\varphi} (u_2-u_1),\\
		\partial_{r \varphi} \tilde u &= \partial_{r \varphi} u_1 + \phi \cdot \partial_{r \varphi} (u_2-u_1) + \phi' \cdot \partial_\varphi (u_2-u_1).
	\end{align*}
	Since $ \delta<1 $, we have $ \abs{\phi},\abs{\phi'},\abs{\phi''}\leq\frac{M}{\delta^2} $. It follows that
	\begin{align}\label{eq_gluingLipschitz}
		\norm{D^\alpha \tilde{u} - D^\alpha u_1}_{C^0(A)} \leq \frac{2M}{\delta^2} \norm{u_1 - u_2}_{C^2(A)} = \frac{2M}{\delta^2} \norm{f_1 - f_2}_{C^2(A,\R^3)},
	\end{align}
	for all multiindices $ \alpha $ with $ \abs{\alpha}\leq 2 $. 
	Furthermore, with the assumption $ \norm{u_i}_{C^2(A)}\leq 1$, we have \com{use triangle inequality in the above relations of $ \partial_{\alpha}\tilde u $ to the derivatives of $ u_i $}
	\begin{align}
		\norm{D^{\alpha}\tilde u}_{C^0(A)} \leq \norm{D^{\alpha}u_1}_{C^0(A)} + \norm{D^{\alpha}\tilde u - D^\alpha u_1}_{C^0(A)}
		&\leq 1 + \frac{4 M}{\delta^2}.\label{eq_gluingLipschitzBoundedDomain}
	\end{align}
	Now, the integrand $ H^2\sqrt{\det G} $ of the Willmore energy, with mean curvature $ H $ and first fundamental form $ G $, is a differentiable function of the first and second derivatives of $ u $. Thus, it is Lipschitz on the bounded domain specified by \eqref{eq_gluingLipschitzBoundedDomain} with some Lipschitz constant $ L(M,\delta)>0 $. Together with \eqref{eq_gluingLipschitz}, we obtain
	\begin{align*}
		\tilde H^2 \sqrt{\det \tilde G} \leq H_1^2\sqrt{\det G_1} + L \frac{2M}{\delta^2} \norm{f_2 - f_1}_{C^2(A,\R^3)}.
	\end{align*}
	With the area $ 4\pi \delta $ of the annulus, we conclude
	\begin{equation*}
		W(\tilde f) \leq W(f_1) + \frac{8\pi L M}{\delta} \norm{f_2 - f_1}_{C^2(A,\R^3)}. \qedhere
	\end{equation*}
\end{proof}
Note that there exist other methods for smoothly gluing two surfaces together that are more efficient in terms of Willmore energy; see \cite{BK_ExistenceWillmore} for an example. The purpose of the gluing defined above is to provide a simple, explicit construction that nonetheless has controlled Willmore energy. 

Let $\lambda>1, R>1, \delta\in(0,1-1/\lambda)\cap(0,R-1)  $. We want to define the $ \delta $-gluing of one half of a catenoid of scale $1/\lambda$ with a sphere of radius $ R $.
First, define 
\begin{align*}
	t_0&:=\frac{\cosh^{-1}(\lambda)}{\lambda},\\
	\theta&:=\arcsin\left(\frac{1}{R} \right),
\end{align*}
and the curves
\begin{align}
	c_1&:[0,\infty)\to \R^2, \qquad t\mapsto\left(\frac{\cosh(\lambda t)}{\lambda} , t-t_0 \right),\label{eq_catenoid}\\
	c_2&:[0,\pi]\to\R^2, \qquad \  t\mapsto R \left( \sin(t), \cos(\theta) - \cos(t)\right).\label{eq_sphere}
\end{align}
The first curve is a catenary and the second the ``east'' half of a circle of radius $R$ and midpoint $ (0,R\cos(\theta))=(0,\sqrt{R^2-1}) $. We have $ c_1(t_0)=c_2(\theta) = (1,0) $. At this intersection point, $ c_1 $ and $ c_2 $ are tangent if and only if $ R = \lambda $. 
Let $ u_1:[1/\lambda, 1+\delta]\to \R $ and $ u_2:[0,R)\to\R $ be the graph parametrizations of $ c_1 $ and $ c_2 $, that is, $ x\mapsto (x,u_i(x)) $ parametrizes a part of $ c_i $. 
We now extend $ u_1,u_2 $ to rotationally symmetric functions on annuli and construct the gluing $ \tilde u : \Ann[1-\delta, 1+\delta]\to\R $ of $ u_1 $ and $ u_2 $ according to Definition~\ref{def_gluing}. Now, define $ g_\beta(x,y):=(x,y,u_\beta(x,y)) $ on $ \Ann[1/\lambda,R) $ with (see also Figure~\ref{fig_catSpheres})
\begin{align*}
	u_\beta(p):=\begin{cases}
		u_1(p) & \text{if } \abs{p} \in [1/\lambda,1-\delta),\\
		\tilde{u}(p) & \text{if } \abs{p} \in [1-\delta,1+\delta],\\
		u_2(p) & \text{if } \abs{p} \in (1+\delta,R).
	\end{cases}
\end{align*}
Let $ D\subseteq\S^2 $ be a closed topological disk and $ f_\beta:D\to\R^3 $ be any smooth immersion such that the interior of $ D $ contains a closed disk $D_1$, $f_\beta|_{D_1}$ parametrizes a closed (round) hemisphere of radius $R$ and $ f_\beta|_{D\setminus D_1} $ parametrizes the same surface as $ g_\beta $ (see Figures~\ref{fig_catSpheres} and~\ref{fig_catSpheresZoom}). We call $f_\beta$ a catenoid sphere of type $\beta$.

To define catenoid spheres of type $\alpha$, we repeat the definition above but with the curve $c_2$ adjusted to be
\begin{align*}
	c_2&:[0,\pi]\to\R^2, \qquad \  t\mapsto R \left( \sin(t), \cos(t) - \cos(\theta) \right).
\end{align*}
This is the ``east'' half of a circle of radius $R$ and midpoint $ (0,-R\cos(\theta)) $. Again, we have $ c_1(t_0)=c_2(\theta) = (1,0) $, but $c_1$ and $c_2$ are never tangent at this intersection point. We proceed analogously with the construction and obtain catenoid spheres $f_\alpha$ of type $ \alpha $ (see Figure~\ref{fig_catSpheres}).

\begin{figure}[htbp]
	\centering
	\hspace*{\fill}
	\begin{minipage}{0.4\textwidth}
		\centering
		\resizebox{0.9\textwidth}{!}{
		\begin{tikzpicture}[scale=1]
			\begin{axis}[
				scale=1,
				axis equal,
				axis lines=middle,
				xtick={-1,1},    
				ytick={0},    
				xmin=0.5,
				xmax=-0.5,
				ymin=-8.5,
				ymax=1.2,
				samples=500,
				legend pos=north west
				]
				\pgfmathsetmacro{\lam}{4}
				\pgfmathsetmacro{\R}{4}
				\pgfmathsetmacro{\tzero}{ln(\lam + sqrt(\lam^2 - 1))/\lam}
				\pgfmathsetmacro{\theta}{rad(asin(1/\R))}
				
				\addplot[myred, thick, domain=0:\tzero,parametric] ({cosh(x*\lam)/\lam},{x-\tzero});
				\addplot[myred, thick, domain=0:\tzero,parametric] ({-cosh(x*\lam)/\lam},{x-\tzero});
				
				\addplot[myblue, thick, domain=\theta:2*pi-\theta,parametric] ({\R*sin(deg(x))},{\R*(-cos(deg(\theta))+cos(deg(x)))});
				\node[anchor=south] at (rel axis cs:0.1,0.9) {type $ \alpha $};
				\node[anchor=south] at (axis cs:0.5,0.1) {\textcolor{myred}{$u_1$}};
				\node[anchor=south] at (axis cs:1.5,0.1) {\textcolor{myblue}{$u_2$}};
				
			\end{axis}
		\end{tikzpicture}
		}
	\end{minipage}
	\hfill
	\begin{minipage}{0.4\textwidth}
		\centering
		\resizebox{\textwidth}{!}{
		\begin{tikzpicture}[scale=1]
			\begin{axis}[
				scale=1,
				axis equal,
				axis lines=middle,
				xtick={-1,1},    
				ytick={0},    
				xmin=0.5,
				xmax=-0.5,
				ymin=-1.2,
				ymax=8.5,
				samples=500,
				legend pos=north west
				]
				\pgfmathsetmacro{\lam}{4}
				\pgfmathsetmacro{\R}{4}
				\pgfmathsetmacro{\tzero}{ln(\lam + sqrt(\lam^2 - 1))/\lam}
				\pgfmathsetmacro{\theta}{rad(asin(1/\R))}
				
				\addplot[myred, thick, domain=0:\tzero,parametric] ({cosh(x*\lam)/\lam},{x-\tzero});
				\addplot[myred, thick, domain=0:\tzero,parametric] ({-cosh(x*\lam)/\lam},{x-\tzero});
				
				\addplot[myblue, thick, domain=\theta:2*pi-\theta,parametric] ({\R*sin(deg(x))},{\R*(cos(deg(\theta))-cos(deg(x)))});
				\node[anchor=south] at (rel axis cs:0.1,0.9) {type $ \beta $};
				\node[anchor=center] at (axis cs:0.5,-1) {\textcolor{myred}{$u_1$}};
				\node[anchor=center] at (axis cs:1.5,0.8) {\textcolor{myblue}{$u_2$}};
				
			\end{axis}
		\end{tikzpicture}
		}
	\end{minipage}
	\hspace*{\fill}
	\caption{The profile curves for the smooth rotationally symmetric catenoid spheres of types $\alpha$ and $\beta$.}\label{fig_catSpheres}
\end{figure}
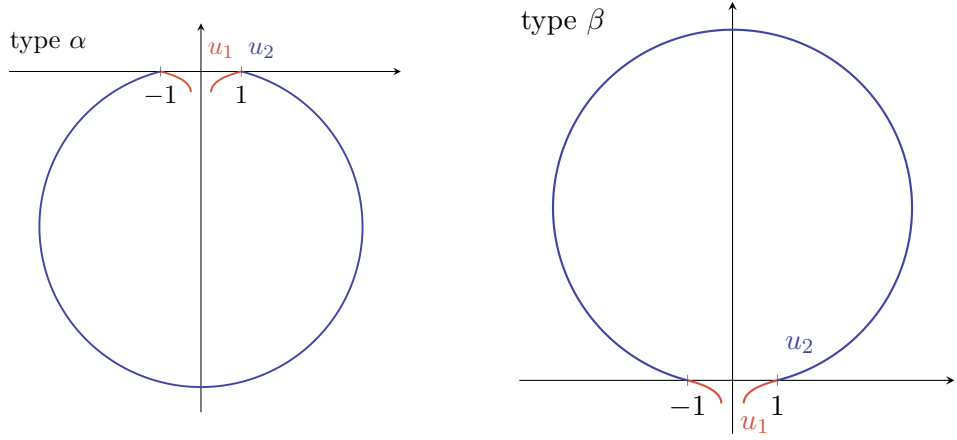
\begin{figure}[htbp]
	\centering
	\begin{tikzpicture}
		\node[anchor=center] at (0,0) {\includegraphics[width=0.4\textwidth]{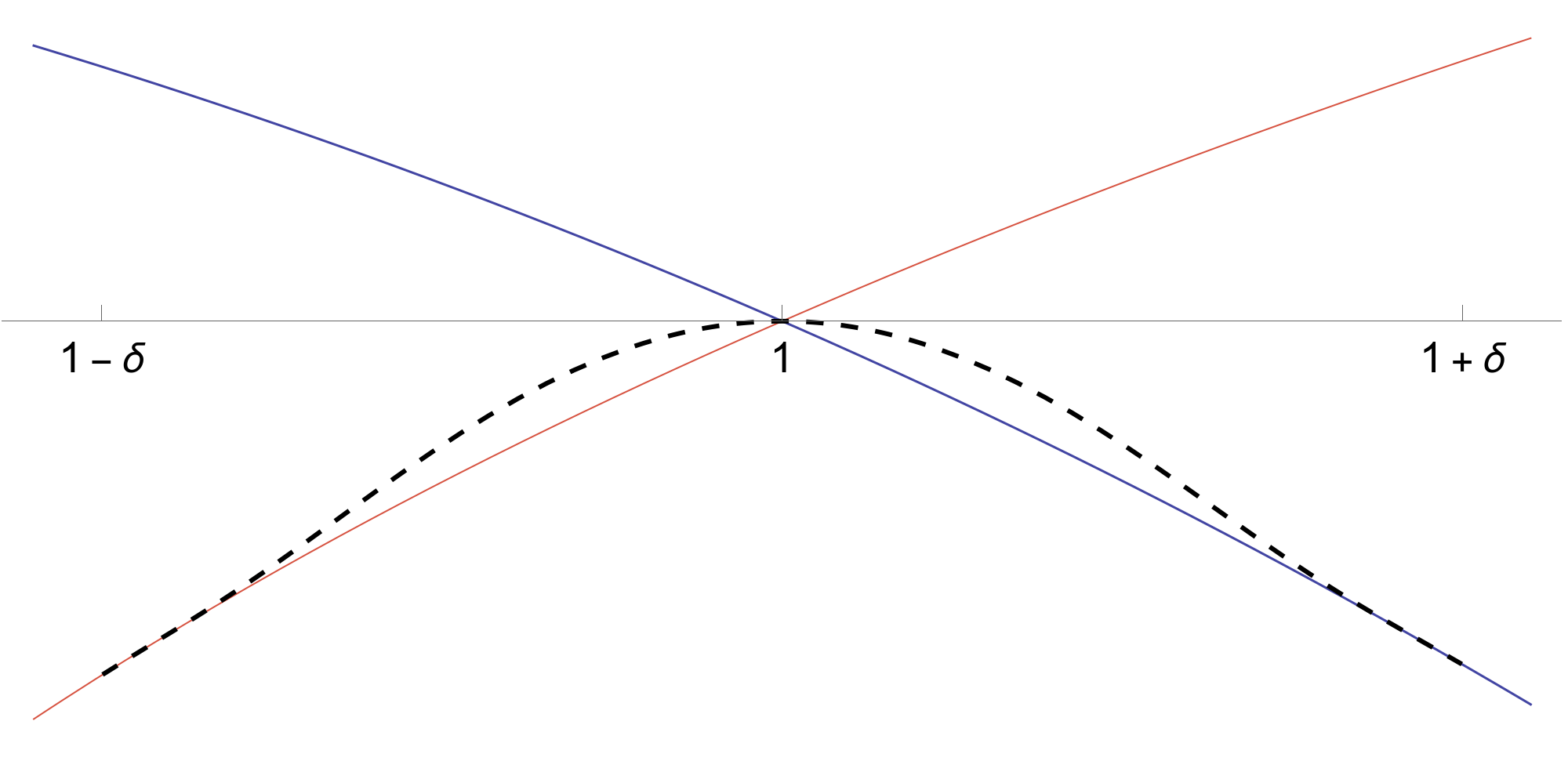}};
		\node[anchor=center] at (-2,-1.2) {\textcolor{myred}{$u_1$}};
		\node[anchor=center] at (2,-1.2) {\textcolor{myblue}{$u_2$}};
		\node[anchor=center] at (0,0.5) {\textcolor{black}{$\tilde u$}};
		\node[anchor=north] at (0,2) {type $ \alpha $};
	\end{tikzpicture}
	\hspace*{0.05\textwidth}
	\begin{tikzpicture}
		\node[anchor=center] at (0,0.18) {\includegraphics[width=0.4\textwidth]{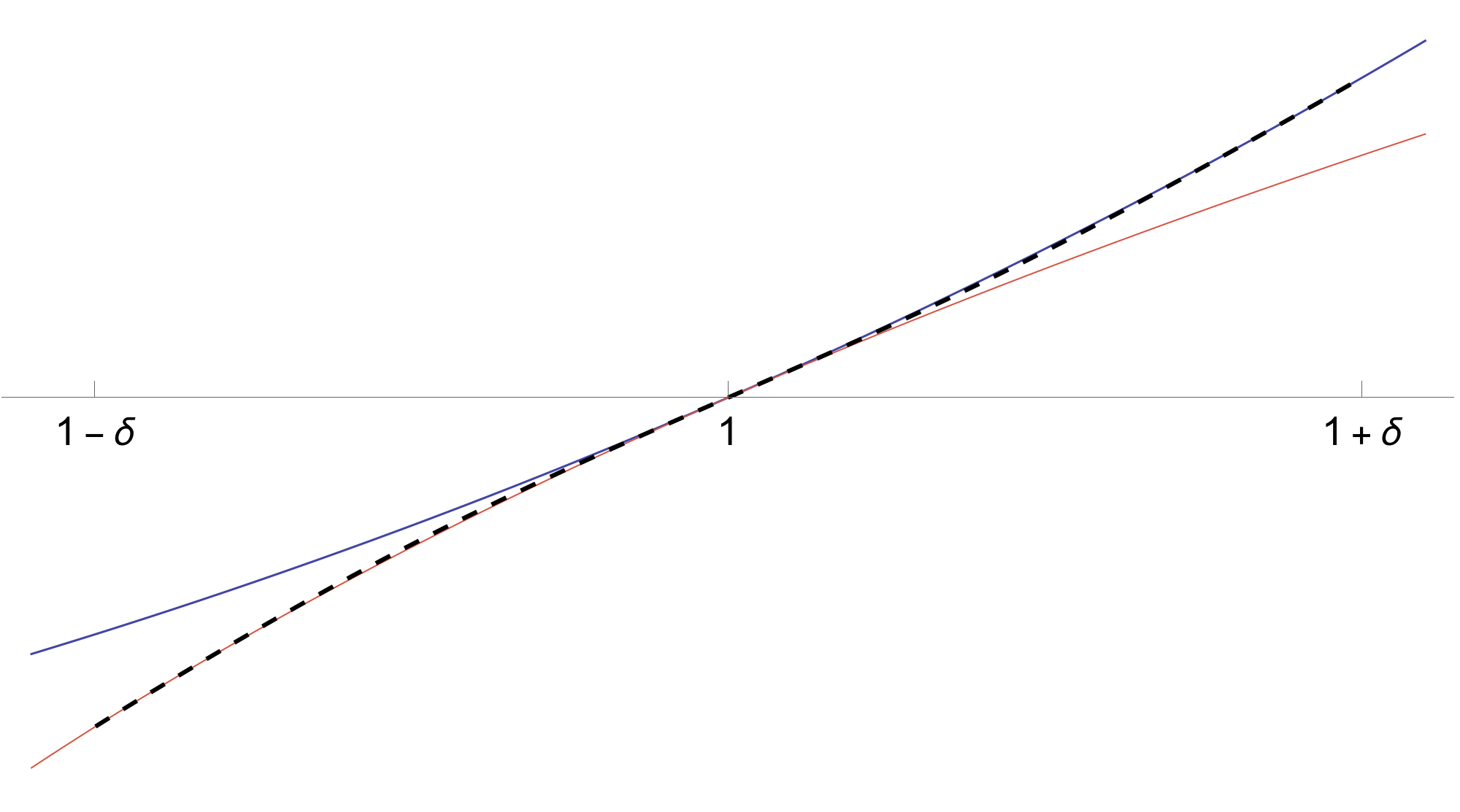}};
		\node[anchor=center] at (-2,-1.2) {\textcolor{myred}{$u_1$}};
		\node[anchor=center] at (2,1.6) {\textcolor{myblue}{$u_2$}};
		\node[anchor=center] at (0,0.5) {\textcolor{black}{$\tilde u$}};
		\node[anchor=north] at (0,2) {type $ \beta $};
	\end{tikzpicture}
	\caption{The gluing region and the functions $ u_1,u_2,\tilde{u} $.}\label{fig_catSpheresZoom}
\end{figure}
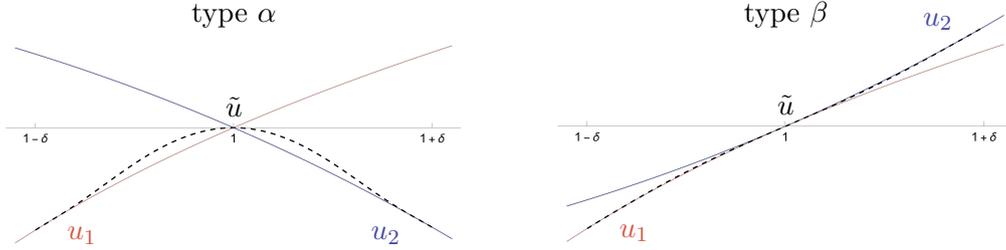

\begin{defn}\label{def_CatSph}
	We define $ \CatSph_\alpha(\lambda,R,\delta) $ (resp.\ $ \CatSph_\beta(\lambda,R,\delta) $) to be the set of all catenoid spheres $ f_\alpha $ (resp.\ $ f_\beta $) of type $ \alpha $ (resp.\ $ \beta $) for given $\lambda>1$, $R>1$ and $\delta\in(0,1-1/\lambda)\cap(0,R-1)$ as defined above. We also set $$ \CatSph(\lambda,R,\delta) := \CatSph_\alpha(\lambda,R,\delta) \cup \CatSph_\beta(\lambda,R,\delta). $$
	We denote by $ W_\alpha(\lambda,R,\delta) $ (resp.\ $ W_\beta(\lambda,R,\delta) $) the Willmore energy of some catenoid sphere $ f\in \CatSph_\alpha(\lambda,R,\delta) $ (resp.\ $ f\in \CatSph_\beta(\lambda,R,\delta) $). We extend these functions by 
	$$ W_\alpha(\lambda,R,0) := W_\beta(\lambda,R,0) := 2\pi \left(1 + \sqrt{1-\frac{1}{R^2}} \right), $$
	which is the Willmore energy of the spherical region in the catenoid sphere before gluing.
\end{defn}

\begin{restatable}{lem}{catenoidSphereShrinking}\label{lem_catenoidSphereShrinking}
	There exists $ \Lambda>0 $ such that the Willmore energy of catenoid spheres satisfies the following.
	\begin{enumerate}[label=\roman*)]
		\item There exists $ \delta_0\in \left(0,\frac12 \right)  $ such that for all $ \delta\in(0,\delta_0) $, the function $ \lambda \mapsto W_\beta(\lambda,\lambda,\delta) $ is increasing on the interval $ [\Lambda,\infty) $.\label{lem_catenoidSphereShrinking_lambdaMonotonicity}
		\item For fixed $ \lambda \geq \Lambda $, we have $$ \lim\limits_{\delta\searrow 0} W_\beta(\lambda,\lambda,\delta) =  W_\beta(\lambda,\lambda,0) < 4\pi. $$\label{lem_catenoidSphereShrinking_deltaToZero}
		\item For fixed $ \delta\in\left(0,\frac12\right) $, we have 
		\begin{align*}
			\lim\limits_{\lambda\to\infty} W_\alpha(\lambda,\lambda,\delta) =
			\lim\limits_{\lambda\to\infty} W_\beta(\lambda,\lambda,\delta) =  4\pi.
		\end{align*}\label{lem_catenoidSphereShrinking_lambdaToInfty}
	\end{enumerate}
\end{restatable}
\begin{proof}
	The Willmore energy of the catenoid vanishes, and hence the Willmore energy of the catenoid sphere is the sum 
	\begin{align*}
		W_{\alpha/\beta}(\lambda,\lambda,\delta) = W^{cap}(\lambda,\delta) + W^{glue}_{\alpha/\beta}(\lambda,\delta),
	\end{align*}
	where 
	\begin{align*}
		W^{cap}(\lambda,\delta)&=2\pi \left( 1 + \sqrt{1 - \frac{(1+\delta)^2}{\lambda^2}} \right)
	\end{align*}
	is the energy of the spherical cap without the gluing region, that is, a round sphere of radius $ \lambda $ with a small cap removed whose boundary circle has radius $ 1+\delta $. And $ W^{glue}_{\alpha/\beta}(\lambda,\delta) $ is the energy of the gluing region (the graph over $ \Ann[1-\delta, 1+\delta] $).
	We observe for the cap (recall Definition~\ref{def_CatSph})
	\begin{align}
		\lim_{\delta\searrow 0} W^{cap}(\lambda,\delta) &= W_{\alpha/\beta}(\lambda,\lambda,0),\label{eq_WCapEstimatedDeltaToZero}\\
		\lim_{\lambda\to\infty} W^{cap}(\lambda,\delta) &= 4\pi.\label{eq_WCapEstimatedLambdaToInfty}
	\end{align}
	Using the explicit parametrizations of the sphere and catenoid and the estimations $ \abs{\phi}\leq 1 $, $ \abs{\phi'} \leq \frac{M}{\delta} $ and $ \abs{\phi''} \leq \frac{M}{\delta^2}$ of the gluing function, one can show that there exist constants $ \Lambda,C>0 $ such that for all $ \lambda\geq \Lambda $ and $ \delta\in(0,\frac12) $, we have
	\begin{align*}
		W^{glue}_{\alpha}(\lambda,\delta)& \leq C\frac{1}{\delta\lambda^2},\\
		W^{glue}_{\beta}(\lambda,\delta)& \leq C \frac{\delta}{\lambda^2},\\
		\partial_\lambda W^{glue}_{\beta}(\lambda,\delta) &\geq - C \frac{\delta}{\lambda^3}.
	\end{align*}
	See Lemma~\ref{lem_catenoidSpheresShrinkingAppendix} for a proof of these estimations.
	Now, the statements~\ref{lem_catenoidSphereShrinking_deltaToZero} and~\ref{lem_catenoidSphereShrinking_lambdaToInfty} follow immediately, and to prove \ref{lem_catenoidSphereShrinking_lambdaMonotonicity}, we set $ \delta_0:=\frac{2\pi}{C} $ to obtain for all $ \lambda\geq \Lambda $ and $ \delta\in(0,\delta_0) $
	\begin{align*}
		\partial_\lambda W^{glue}_{\beta}(\lambda,\delta) 
		&\geq - C \frac{\delta}{\lambda^3} \\
		&> - \frac{2\pi}{\lambda^3} \\
		&> -\frac{2\pi(1+\delta)^2}{\lambda^2 \sqrt{\lambda^2-(1+\delta)^2}}\\
		&= -\partial_\lambda W^{cap}(\lambda,\delta),
	\end{align*}
	which yields $ \partial_\lambda W_\beta(\lambda,\lambda,\delta)>0  $ and thus statement~\ref{lem_catenoidSphereShrinking_lambdaMonotonicity}.
\end{proof}
\begin{rem}
	In contrast to Lemma~\ref{lem_catenoidSphereShrinking}~\ref{lem_catenoidSphereShrinking_deltaToZero}, we have $ W_\alpha(\lambda,\lambda,\delta)>4\pi $ for all admissible $ \lambda $ and $ \delta $. This is due to the fact that any catenoid sphere of type $ \alpha $ builds an immersed sphere of the family $ J $ with a copy of its own. And these immersions all have Willmore energy greater than $ 8\pi $ (see Remark~\ref{rem_JW>8pi}).
\end{rem}

\subsection{Global Classification of Singularities below $ 12\pi $}\label{sec_surgery}
The aim of this section is to prove Theorem~\ref{thm_globalSing} by giving a global classification of singularities below $ 12\pi $. By this we do not mean to classify all singular Willmore flows at their singular time, but rather to find a small set of surfaces to which every Willmore flow can be deformed without increasing the Willmore energy. This approach relies on the classification of blow-ups of singular Willmore flows below $ 16\pi $ in \cite{LN}. The blow-up procedure was introduced in \cite{KS_SmallInitial}. We refer to the descriptions in \cite[Section~3.3]{KS_Survey} and \cite[Section~4]{LN}. Consider a maximal Willmore flow $ f:\S^2\times [0,T)\to\R^3 $ with $ T\in(0,\infty] $. If $ f_t $ does not converge to a round sphere as $ t\to T $, then there exist sequences $ t_j\in[0,T) $, $ r_j>0 $, $ x_j\in\R^3 $ with $ t_j\to T $ and a rescaled flow
\begin{align}\label{eq_blowup}
	f^j:\S^2 \times \left[ -\frac{t_j}{r_j^4}, \frac{T-t_j}{r^4_j} \right) \to \R^3, \qquad f^j(p,\tau) = \frac{1}{r_j}\left( f(p,t_j+r^4_j\tau) -x_j \right).
\end{align}
\begin{thmCite}\label{thm_LN}\text{(\cite[Theorem~1.6]{LN}).}
	If $ W(f_0)<16\pi $, then $ f^j(\cdot,0) $, as above, converges locally smoothly after appropriate reparametrization to either a catenoid, trinoid or Enneper's minimal surface. If $ W(f_0)\leq 12\pi $, then only the catenoid case appears.
\end{thmCite}
Here, \textit{trinoid} refers to any complete genus zero minimal surface of total curvature $ -8\pi $ with three ends, as classified in \cite[Theorem~3]{Lopez_CompleteMSTotCurv}.
The details of the convergence in this case are as follows (cf. \cite[Section~4]{KS_SmallInitial}, \cite[Section~3.3]{KS_Survey}). There exist a two-dimensional manifold $ \hat{\Sigma}=\S^2\setminus\{p_1,\dots,p_k\} $ and a limit immersion $ \hat f : \hat \Sigma \to \R^3 $, that is, a catenoid ($ k=2 $), trinoid ($ k=3 $) or Enneper's surface ($ k=1 $). Define the domains $ S_j:=(f^j)^{-1}(B_j(0)) $ and $ \hat S_j:=\hat f^{-1}(B_j(0)). $
There exist open sets $ U_j\subseteq \S^2 $ such that for all $ R>0 $, there exists $ j_R\in\N $ and we have for all $ j\geq j_R $ that $ f^{-1}(B_R(0)) \subseteq U_j $. There exist diffeomorphisms $ \varphi_j: \hat S_j \to U_j $
and a function $ u_j\in C^\infty(\hat{S}_j,\R) $ that satisfy, with some fixed normal $  \hat \nu $ of $ \hat f $, for all $ k\in\N $
\begin{align*}
	f^j \circ \varphi_j = \hat{f} + u_j \hat \nu,\\
	\lim_{j\to \infty} \norm{\hat \nabla ^k u_j}_{L^\infty\left(\hat S_j\right)} = 0.
\end{align*}
This implies the following. For all $ \lambda>0 $, the functions 
$ g_j := f^j \circ \varphi_j : \hat f^{-1}(\overline{B_\lambda(0)}) \to \R^3 $ are defined for sufficiently large $ j\in\N $ and satisfy for all $ k\in\N $
\begin{align}
	\lim_{j\to\infty}\norm{g_j - \hat f}_{C^k\left(\hat S_\lambda,\R^3\right)} =0.
\end{align}
Here and in the following, the $ C^k $ norms of immersions defined on subsets of $ \S^2 $ are defined by fixing charts of $ \S^2 $.
\begin{rem}
	In the following, we construct a regular homotopy with controlled Willmore energy. By scale invariance of the Willmore energy, we can scale any homotopy $ f_t $ by a factor $ \lambda_t $ without changing the Willmore energy. Furthermore, the set of diffeomorphisms on $ \S^2 $ has two homotopy classes, orientation-preserving and orientation-reversing reparametrizations \cite{Smale_DiffS2}. Thus, we can always find a regular homotopy of constant Willmore energy from one immersion $ f $ to a reparametrized $ f \circ \varphi $ of the same orientation. In the following, we implicitly scale or reparametrize by a homotopy whenever needed.
\end{rem}
\begin{prop}\label{prop_gluingA}
	Let $ f_0:\S^2\to\R^3 $ be a smooth immersion with $ W(f_0)\leq 12\pi $ and $ \delta\in\left(0,\frac14\right) $. Then, there exists $ \Lambda>0 $ such that for all $ \lambda,R\geq \Lambda $, there exists a regular homotopy $ H:\S^2\times [0,1]\to \R^3 $ satisfying the following.
	\begin{enumerate}[label=\roman*)]
		\item $ H_0=f_0 $.
		\item $ W(H_t)<W(f_0) $ for all $ t\in (0,1] $.
		\item Either $ H_1 $ is a round sphere or there exist two closed topological disks $ S_1,S_2\subseteq \S^2 $ with $ S_1\cup S_2 = \S^2 $ and $ S_1\cap S_2 =\partial S_1 = \partial S_2$ such that $ H_1|_{S_1},H_1|_{S_2} \in \CatSph(\lambda,R,\delta) $.
	\end{enumerate} 
\end{prop}
\begin{proof}
	Let $ f:\S^2\times [0,T) \to \R^3 $ be the maximal Willmore flow with initial surface $ f_0 $ and $ T\in (0,\infty] $. Fix some $ s_0\in (0,T)  $ and set 
	$$ \omega:=W(f_0)-W(f_{s_0})>0. $$ 
	Let $ A:=\Ann[1-\delta,1+\delta] $ and $\rho>0 $ sufficiently small such that by Lemma~\ref{lem_gluingLemma} the gluing excess Willmore energy for a $ \delta $-gluing between two graphs $ g_1,g_2 \in C^2(A,\R^3) $ with $ \norm{g_i}_{C^2(A)}\leq 1 $ and $ \norm{g_1 - g_2}_{C^2(A,\R^3)}<\rho $ is less than $ \frac{\omega}{5} $.
	
	\underline{Step 1:} Prepare catenoids and spheres.\\
	Fix $ \Lambda>0 $ sufficiently large such that for $ \lambda,R\geq \Lambda $, catenoids of scale $ \frac{1}{\lambda} $, rotation axis $ (0,0,1) $ and center $ (0,0,c) $ with $ c\in\left[0,\lambda^{-1}\arccosh(\lambda)\right] $, and spheres of radius $ R $ and center $ (0,0,m) $ with $ m\in\left[\sqrt{R^2-1},R\right] $ (see \eqref{eq_catenoid} and \eqref{eq_sphere} for $ c=\lambda^{-1}\arccosh(\lambda) $ and $ m=\sqrt{R^2-1} $) have graph parametrizations $ u_c, u_s: A \to \R $ over $ A\times\{0\} $, respectively, which satisfy
	\begin{align}
		\norm{u_c}_{C^2\left(A\right)}<\frac{\rho}{8},\label{eq_catCloseToFlatDisk}\\
		\norm{u_s}_{C^2\left(A\right)}<\frac{\rho}{8}.\label{eq_sphereCloseToFlatDisk}
	\end{align}

	\underline{Step 2:} Fix domains and notation for gluing.\\
	Let $ S_1,S_2\subseteq \S^2 $ be the northern and southern closed hemispheres and fix a closed annulus $ C\subseteq \S^2 $ and two disjoint closed disks $ D_1\subseteq S_1,D_2\subseteq S_2 $ such that they overlap in annuli $ C_i:=D_i \cap C $ with nonempty interior (see Figure~\ref{fig_domainsOnS2}).
	To make the gluings of immersions along their graphs over the annulus $ A $ well-defined in the following, we fix once diffeomorphisms $ \psi^+_i,\psi^-_i: A \to C_i $ that are orientation preserving and reversing, respectively.
	We say that an immersion $ h:C_i\to \R^3 $ is a graph over $ A\times\{0\}\subseteq \R^3 $ if $ h\circ \psi^+_i $ or $ h\circ \psi^-_i $ is a graph on $ A $, that is, $ h\circ \psi^{\pm}(p) = (p,u_h(p)) $ for some $ u_h:A\to\R $\com{$ \psi^+ $ or $ \psi^- $ depending on the orientation of $ h $}. In the following, before each gluing, we implicitly assume that by orientation-preserving reparametrizations, the involved immersions are graphs over $ A $ in the way just defined. The resulting gluing is always an immersion. Outside of the gluing region, we recover the involved immersions and inside, we obtain a graph, in particular an immersion. \com{Note that we always glue immersed disks to each other along an annulus overlap so that the obtained gluing is an immersed sphere and not of other topology. 
	}
	\begin{figure}[H]
		\centering
		\resizebox{0.8\textwidth}{!}{
			\begin{tikzpicture}
				\node[anchor=center] at (-0.2,-0.25) {\includegraphics[width=10cm]{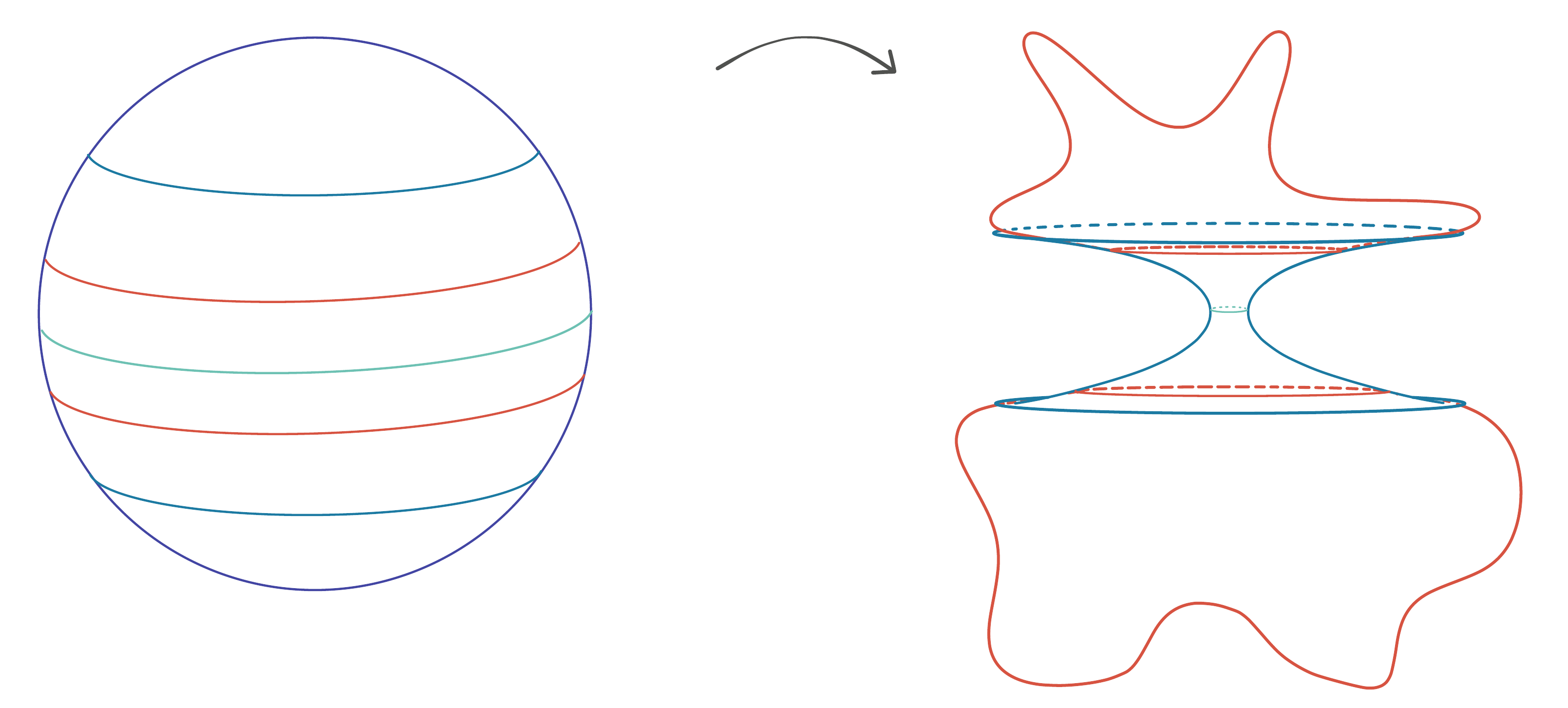}};
				
				\def\ampl{7pt}
				
				\foreach \cx/\cy/\length/\lab in 
				{
					-5/0/2.3/$C$
				} 
				{
					\draw[decorate,decoration={brace,amplitude=\ampl}, mypastel] (\cx,\cy-0.5*\length) -- (\cx,\cy+0.5*\length);
					\node[anchor=east] at (\cx-0.1,\cy) {\small \textcolor{mypastel}{\lab}};
				}
				
				\foreach \cx/\cy/\length/\lab in 
				{
					-5.8/ 0.9/1.7/$S_1$,
					-5.8/-0.9/1.7/$S_2$
				} 
				{
					\draw[decorate,decoration={brace,amplitude=\ampl}, black] (\cx,\cy-0.5*\length) -- (\cx,\cy+0.5*\length);
					\node[anchor=east] at (\cx-0.1,\cy) {\small \textcolor{black}{\lab}};
				}
				
				\foreach \cx/\cy/\length/\lab in 
				{
					-1.4/ 1.1/1.4/$D_1$,
					-1.4/-1.1/1.4/$D_2$
				} 
				{
					\draw[decorate,decoration={brace,mirror,amplitude=\ampl}, myred] (\cx,\cy-0.5*\length) -- (\cx,\cy+0.5*\length);
					\node[anchor=west] at (\cx+0.15,\cy) {\small \textcolor{myred}{\lab}};
				}
				\node[anchor=center] at (0,2) {\small \textcolor{black}{$f_{t_0}$}};
				\node[anchor=center] at (-3.2,0.45) {\small \textcolor{myblue}{$C_1$}};
				\node[anchor=center] at (-3.2,-1) {\small \textcolor{myblue}{$C_2$}};
			\end{tikzpicture}
		}
		\caption{The introduced domains on $ \S^2 $ and the immersion $ f_{t_0} $.}\label{fig_domainsOnS2}
	\end{figure}
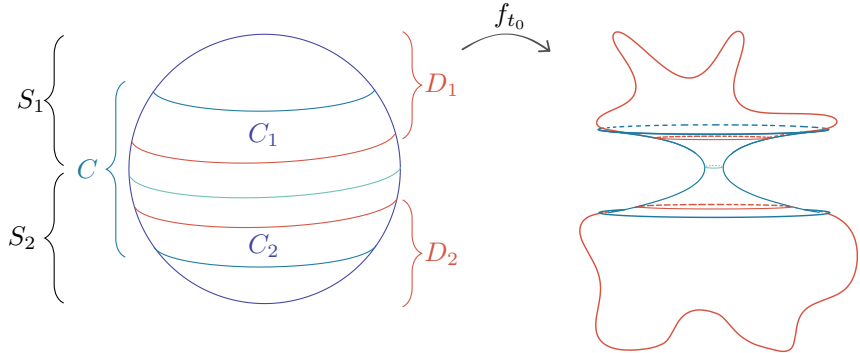
	\underline{Step 3:} Fix a time right before the singular time.\\
	Fix $ \lambda\geq \Lambda $ and $ \hat f: C \to \R^3 $ that is a parametrization of appropriate orientation\com{which is compatible with $ f_{t_0} $} of a (precise) catenoid piece of scale $ \frac{1}{\lambda} $ in $ \overline{B_2(0)}\subseteq \R^3 $ 
	such that $ \hat f|_{S_1} $ parametrizes the part in the upper half-space and $ \hat f|_{C_i} $ are graphs over $ A\times \{0\} $. Then, by Theorem~\ref{thm_LN}, there exists $ t_0\in (s_0,T) $ such that, with appropriate reparametrization and rescaling of $ f_{t_0} $, the error $ \rho_1:=\norm{f_{t_0} - \hat{f}}_{C^2(C,\R^3)} $ is sufficiently small. We choose it so small that $ f_{t_0}|_{C_i} $ are graphs over $ A\times \{0\} $ (see Figure~\ref{fig_domainsOnS2}) and their graph parametrizations $ v^i_{t_0} $ of $ f_{t_0}|_{C_i} $ and $ \hat u^i $ of $ \hat f|_{C_i} $ in the upper half-space ($ i=1 $) and in the lower half-space ($ i=2 $) satisfy 
	\begin{align}
		\norm{v^i_{t_0}-\hat u^i}_{C^2\left(A\right)}<\frac{\rho}{8}.\label{eq_blowUpCloseToCatenoid}
	\end{align}
	Furthermore, $ \rho_1 $ is chosen such that all $ g:C\to \R^3 $ with $ \norm{\hat f - g}_{C^2(C,\R^3)}\leq \rho_1 $ satisfy 
	\begin{equation}
		W(g)<W(\hat{f})+\frac{\omega}{5}=\frac{\omega}{5}.\label{eq_straightLineToCat}
	\end{equation} 
	
	\underline{Step 4:} Construct the final homotopy.\\
	We construct three regular homotopies
	\begin{align*}
		H^1&:\S^2\times[0,1]\to\R^3,\\
		H^c&:C \times[0,1]\to\R^3,\quad (p,t)\mapsto f_{t_0}(p)+t\left(\hat{f}(p)-f_{t_0}(p)\right),\\
		H^2&:\S^2\times[0,1]\to\R^3,
	\end{align*}
	where $ H^1,H^2 $ are defined later (see Step 5) and satisfy all of the following (see also Figure~\ref{fig_catenoidSurgery}).
	\begin{enumerate}[label=\roman*)]
		\item $ H^i_0|_{D_i}= f_{t_0}|_{D_i}  $,\label{enum_WFfirst}
		\item $ W(H^i_0) < W(f_{t_0}|_{D_i}) + \frac\omega5, $\label{enum_WFlowEnergy}
		\item $ t\mapsto W(H^i_t)$ is nonincreasing,\label{enum_WFdec}
		\item $ H^i_t|_{C_i} $ is a graph over $ A\times\{0\} $ for all $ t\in[0,1] $\label{enum_WFgraph},
		\item The graph parametrizations $ u^i_t:A\to \R $ of $ H^i_t $ over $ A\times\{0\} $ satisfy\label{enum_WFcloseToFlatDisk}
		$$ \norm{u^i_t}_{C^2(A)} <\frac{\rho}{2},$$
		\item $ H^i_1 $ is a round sphere of radius $ R $ and center $ \left(0,0,\lambda^{-1}\arccosh(\lambda) \pm \sqrt{R^2-1}\right) $.\label{enum_WFfinalSphere}
		\label{enum_WFlast}
	\end{enumerate} 
	We glue them together to a regular homotopy $ H:\S^2\times[0,1]\to\R^3 $ that for each time $ t\in[0,1] $ is defined on $ D_i $ to be the $ \delta $-gluing of $ H^i_t|_{D_i} $ with $ H^c_t $ along the graph parametrizations over $ A\times\{0\} $, for $ i=1,2 $. And on $ \S^2\setminus (D_1 \cup D_2) $, it is defined to agree with $ H^c_t $.
	The graph parametrizations $ u^{c,i}_t:A\to \R $ of $ H^c_t $ over $ A\times\{0\} $ satisfy with appropriate reparametrizations $ u^{c,i}_t = v^i_{t_0} + t\left(\hat u^i - v^i_{t_0}\right) $ and thus with \eqref{eq_blowUpCloseToCatenoid} and \eqref{eq_catCloseToFlatDisk}
	\begin{align*}
		\norm{u^{c,i}_t}_{C^2(A)} &\leq \norm{v^i_{t_0} - \hat u^i}_{C^2(A)} + \norm{\hat u^i}_{C^2(A)} + t \norm{\hat u^i - v^i_{t_0}}_{C^2(A)}\\
		&< \frac{\rho}{8} + \frac{\rho}{8} + \frac{\rho}{8} < \frac{\rho}{2}.
	\end{align*}
	Thus, with \ref{enum_WFcloseToFlatDisk}, we have $ \norm{u^c_t - u^i_t}_{C^2(A,\R)} < \rho $ and the gluing excess energy of each of the two $ \delta $-gluings remains below $ \frac{\omega}{5} $ for all times. Furthermore, by \eqref{eq_straightLineToCat}, we have $ W\left(H^C_t\right)< \frac{\omega}{5} $ and obtain with \ref{enum_WFdec} and \ref{enum_WFlowEnergy}
	\begin{align*}
		W(H_t)&<W(H^1_t)+W(H^C_t)+W(H^2_t) + \frac{2}{5}\omega\\
		&<W(H^1_t)+W(H^2_t)+\frac35\omega\\
		&\leq W(H^1_0)+W(H^2_0)+\frac35\omega\\
		&< W(f_{t_0})+\omega\\
		&< W(f_0).
	\end{align*}
	Now, a concatenation of $ f:\S^2\times[0,t_0]\to\R^3 $ with $ H $ is continuous since \ref{enum_WFfirst} ensures that $ H_0=f_{t_0} $. We obtain a regular homotopy that ends in $ H_1 $ whose Willmore energy remains below $ W(f_0) $. Since $ H^c_1 = \hat f $ is a catenoid of the form \eqref{eq_catenoid} but shifted by $ (0,0,\lambda^{-1}\arccosh(\lambda)) $, we conclude with \ref{enum_WFfinalSphere} that $ H_1|_{S_1},H_1|_{S_2}\in\CatSph(\lambda,R,\delta) $.
	
	\underline{Step 5:} It remains to define $ H^i $, $ i=1,2 $, which satisfy \ref{enum_WFfirst}--\ref{enum_WFlast}.\\ 
	We fix $ i=1 $ and hence consider the graphs in the upper half-space of $ \R^3 $. $ H^2 $ is constructed analogously considering the graphs in the lower half-space. For the purpose of the remaining proof, we shift all the immersions by $ -\left(0,0,\lambda^{-1}\arccosh(\lambda)\right) $ for the catenoid $ \hat f $ to be of the form \eqref{eq_catenoid} and the sphere obtained in the very end to be of the form \eqref{eq_sphere}.
	The regular homotopy $ H^1 $ will be obtained as a concatenation of multiple steps (see also Figure~\ref{fig_catenoidSurgery})
	\[
	\setlength{\arraycolsep}{ 0.167 em}
	\begin{array}{rlrl}
		\mu     & \mapsto \mu F_0,             & \mu  & \in [1, \bar\mu], \\
		t       & \mapsto \bar\mu F_t,         & t   & \in [0, 1], \\
		\lambda & \mapsto \lambda \bar\mu F_1, & \lambda & \in [1, \bar\lambda], \\
		s       & \mapsto \bar\lambda \bar\mu F_1 + sy, & s & \in [0,1],
	\end{array}
	\]
	the details of which will be laid out in the following.
	\begin{figure}[H]
		\centering
		\begin{tikzpicture}
			\node[anchor=south] at (0,0) {\includegraphics[width=0.95\textwidth]{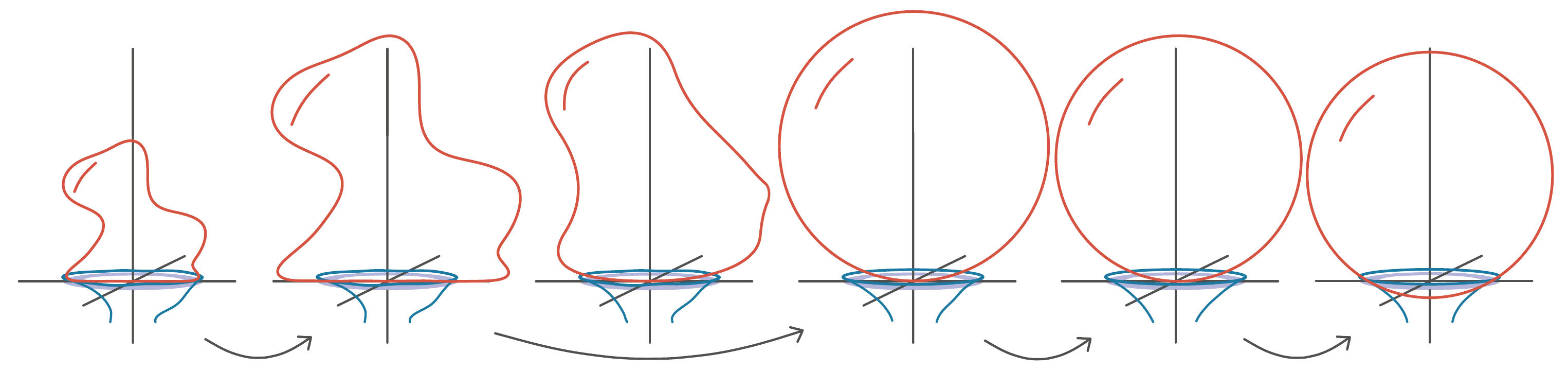}};
			
			\def\dx{0.158*\textwidth};
			
			\node[anchor=north east] at (-2.65*\dx,1.05) {\scalebox{0.4}{ \textcolor{myblue}{$A\times \{0\}$} }};
			\node[anchor=north east] at (-2.8*\dx,2) {\small \textcolor{myred}{$H^1_t$}}; 
			\node[anchor=north east] at (-2.8*\dx,0.7) {\small \textcolor{mypastel}{$H^c_t$}};
			
			\node[anchor=north] at (-2*\dx,0.2) {scaling};
			\node[anchor=north] at (-0.5*\dx,0.2) {Willmore flow};
			\node[anchor=north] at (1*\dx,0.2) {scaling};
			\node[anchor=north] at (2*\dx,0.2) {translation};
		\end{tikzpicture}
		
		\caption{Simultaneous regular homotopies to a catenoid piece and a round sphere glued to each other at each time resulting in a catenoid sphere of type $ \beta $ (gluing not displayed).}\label{fig_catenoidSurgery}
	\end{figure}
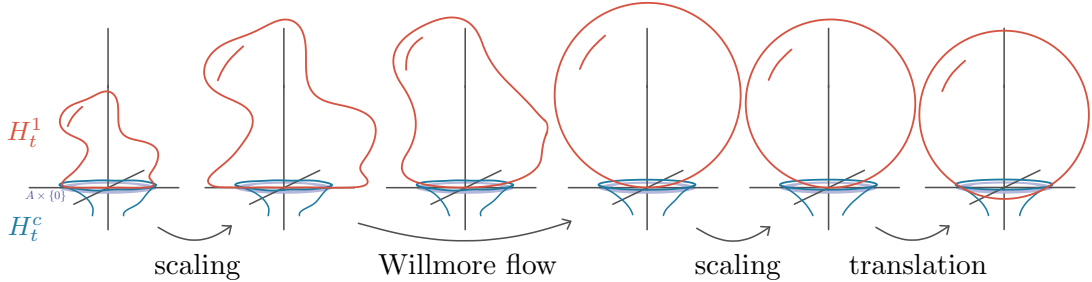
	\underline{Step 5.1:} Construct $ H^1_0 $ by gluing in a flat disk and start the Willmore flow.\\
	Let $ d_1 $ be an appropriate parametrization of the flat disk $ \overline{B_{\frac{1+\delta}{2}}(0)}\times\{0\} $ with graph parametrization $ u_d=0 $.
	Let $ H^1_0:\S^2\to \R^3 $ be the gluing of $ d_1 $ with $ f_{t_0}|_{D_1} $ along the graphs over $ \Ann\left[\frac{1-\delta}{2},\frac{1+\delta}{2}\right] $. Appropriate choice of $ \lambda,\rho_1 $ can ensure that \eqref{eq_catCloseToFlatDisk} and \eqref{eq_blowUpCloseToCatenoid} equally apply to $ \Ann\left[\frac{1-\delta}{2},\frac{1+\delta}{2}\right] $ instead of $ A $. We obtain
	\begin{align}
		\norm{v^1_{t_0}}_{C^2\left(\Ann \left[\frac{1-\delta}{2},1+\delta\right]\right)} &\leq \norm{v^1_{t_0} - \hat u^1}_{C^2\left(\Ann \left[\frac{1-\delta}{2},1+\delta\right]\right)} + \norm{\hat u^1}_{C^2\left(\Ann \left[\frac{1-\delta}{2},1+\delta\right]\right)} \notag\\
		&< \frac\rho4,\label{eq_t0CloseToDisk}
	\end{align}
	By doubling the scale, this can be made a standard $ \delta $-gluing according to Lemma~\ref{lem_gluingLemma} and thus the gluing excess energy is less than $ \frac{\omega}{5} $. We obtain \ref{enum_WFfirst} and \ref{enum_WFlowEnergy}.
	Defining $ H^2_0 $ analogously, these initial surfaces $ H^1_0,H^2_0 $ have Willmore energy
	\begin{align*}
		4\pi \leq W(H^i_0)< W(f_{t_0}|_{D_i}) + \frac{\omega}{5},
	\end{align*}
	and therefore
	\begin{align*}
		W(H^1_0) &< W(f_{t_0}|_{D_1}) + \frac{\omega}{5}\\
		&\leq W(f_{t_0}) - W(f_{t_0}|_{D_2}) + \frac{\omega}{5}\\
		&< W(f_{t_0}) - \left( W(H^2_0) -\frac{\omega}{5} \right) + \frac{\omega}{5}\\
		&\leq W(f_{t_0}) -4\pi + \frac{2}{5}\omega \\
		&< 12\pi - \omega -4\pi +\frac{2}{5}\omega\\
		&<8\pi.
	\end{align*}
	Therefore, by Theorem~\ref{thm_KS8pi}, the Willmore flow $ \S^2\times[0,\infty)\to\R^3 $ with initial surface $ H^1_0 $ converges smoothly to a round sphere. By monotonically rescaling the time interval $ [0,\infty)\to[0,1) $, we obtain a smooth regular homotopy $ F:\S^2\times [0,1] \to\R^3 $ where we set $ F_1 $ to be the limit round sphere. Now, $ F $ is a regular homotopy of decreasing Willmore energy from $ F_0=H^1_0 $ to a round sphere. In the following, we modify $ F $ only by translation, rotation, scaling and reparametrization. Hence, \ref{enum_WFdec} will remain satisfied.
	
	\underline{Step 5.2:} Increase scale to ensure low gluing energy.\\
	Let $ p:=d_1^{-1}(0) $ denote the preimage of the center of the flat disk. Then, its tangent space $ T_pF_0 $ is the $ x $-$ y $-plane $ \R^2\times\{0\} $ and by translation and rotation of $ F_t $, we ensure that $ F_t(p)=0 $ and $ T_pF_t = \R^2\times\{0\} \subseteq \R^3$ for all $ t\in[0,1] $. 
	The Willmore flow might converge to a very small sphere or temporarily increase curvature locally, which might violate \ref{enum_WFcloseToFlatDisk}. We counteract this by scaling. First, note that for $ r\in(0,1+\delta) $, any $ u\in C^2(B_r(0)) $ with $ u(0)=0 $ and $ \nabla u (0) = 0 $ satisfies by Taylor expansion
	\begin{align}\label{eq_C2normBoundBySecDer}
		\norm{u}_{C^2(B_r(0))} \leq \left(\frac{r^2}{2} + r + 1\right)\norm{D^2u}_{C^0(B_r(0))} \leq 4 \norm{D^2u}_{C^0(B_r(0))} .
	\end{align}
	The norm of the second fundamental form $ (p,t)\mapsto \norm{A_t(p)} $ of $ F_t $ in $ p\in \S^2 $ assumes a maximum $ A_{\max} $ on the compact spacetime $ \S^2\times[0,1] $. Therefore, by Lemma~\ref{lem_uniformGraphRadius}, there exists a radius $ r\in(0,1) $ such that for all $ t\in[0,1] $, the immersion $ F_t $ is locally a graph over $ B_r(0) \subseteq T_pF_t $ with graph parametrization $ u_t $ satisfying 
	\begin{align*}
		\norm{D^2u_t}_{C^0(B_r)} \leq 5 A_{\max}.
	\end{align*}
	We now insert the scaling homotopy $ \mu\mapsto \mu F_0 $, $ \mu\in\left[1,\bar{\mu}\right], $ where $ \bar \mu> \frac{1+\delta}{r} $ and $ \bar\mu > \frac{40}{\rho}A_{\max} $. 
	During this scaling, \ref{enum_WFcloseToFlatDisk} is satisfied since \eqref{eq_t0CloseToDisk} and \eqref{eq_C2normBoundBySecDer}
	imply for the graph parametrization $ w_\mu(x) = \mu u_0\left(\frac{x}{\mu}\right) $ of $ \mu F_0 $
	\begin{align*}
		\norm{w_\mu}_{C^2(A)} &\leq \norm{w_\mu}_{C^2\left(\overline{B_{1+\delta}(0)}\right)} \\
		&\leq 4 \norm{D^2 w_\mu}_{C^0\left(\overline{B_{1+\delta}(0)}\right)} \\
		&=4 \frac1\mu \norm{D^2 u_0}_{C^0\left(\overline{B_{\frac{1+\delta}{\mu}}(0)}\right)} \\
		&\leq 4 \frac1\mu \norm{D^2 u_0}_{C^0\left(\overline{B_{1+\delta}(0)}\right)} \\
		&=4 \frac1\mu \norm{D^2 u_0}_{C^0\left(\Ann \left[\frac{1-\delta}{2},1+\delta\right]\right)} \\
		&=4 \frac1\mu \norm{D^2 v^1_{t_0}}_{C^0\left(\Ann \left[\frac{1-\delta}{2},1+\delta\right]\right)} \\
		&<4 \frac1\mu \frac\rho8 < \frac\rho2.
	\end{align*}
	Furthermore, we have ensured that after the scaling the graph parametrizations $ z_t(x) = \bar{\mu} u_t\left(\frac{x}{\bar{\mu}}\right)  $ of the scaled Willmore flow $ \bar{\mu} F_t $ over $ A\times\{0\} $ satisfy \ref{enum_WFcloseToFlatDisk}, that is,
	\begin{align*}
		\norm{z_t}_{C^2(A)}&\leq \norm{z_t}_{C^2\left(\overline{B_{1+\delta}(0)}\right)} \\
		&\leq 4 \norm{D^2z_t}_{C^0\left(\overline{B_{1+\delta}(0)}\right)}\\
		&=4 \frac{1}{\bar{\mu}} \norm{D^2u_t}_{C^0\left(\overline{B_{\frac{1+\delta}{\bar \mu}}(0)}\right)}\\
		&\leq 4 \frac{1}{\bar \mu} \norm{D^2u_t}_{C^0\left(B_r(0)\right)}\\
		&\leq 4 \frac{1}{\bar \mu} 5 A_{\max}\\
		&< \frac\rho2.
	\end{align*} 
	\underline{Step 5.3:} Append final correcting homotopies.\\
	During this regular homotopy, we reparametrize $ \bar \mu F_t $ appropriately such that $ \bar \mu F_t|_{C_i} $ is a graph over $ A\times\{0\} $ for \ref{enum_WFgraph} to hold. 
	Finally, we append a scaling regular homotopy that brings the round sphere to any given radius $ R\geq \Lambda $. It is now tangent to the $ x $-$ y $-plane and centered at $ (0,0,R) $ (type $ \beta $) or $ (0,0,-R) $ (type $ \alpha $). Then, we append a translation that shifts the center of the sphere from $ \pm(0,0,R) $ to $ \pm\left(0,0,\sqrt{R^2-1}\right) $. The translation brings the sphere into the form as described in Definition~\ref{def_CatSph}. In the case of the sphere lying entirely above the horizontal plane (type $ \beta $), it is now tangent to the catenoid $ H^c_1 $ in the unit circle $ \S^1 \times \{0\} $. By \eqref{eq_sphereCloseToFlatDisk}, these two steps do not violate \ref{enum_WFcloseToFlatDisk}. Reversing the shift of all surfaces by $ -\left(0,0,\lambda^{-1}\arccosh(\lambda)\right) $, we observe that the sphere satisfies \ref{enum_WFfinalSphere}.
\end{proof}

\begin{proof}[Proof of Theorem~\ref{thm_globalSing}]
	Fix $ \Lambda,\delta_0 $ from Lemma~\ref{lem_catenoidSphereShrinking}~\ref{lem_catenoidSphereShrinking_lambdaMonotonicity} and $ \delta\in(0,\delta_0) $ sufficiently small such that $ \eps:= 4\pi - W_\beta(\Lambda,\Lambda,\delta) $ is positive (see Lemma~\ref{lem_catenoidSphereShrinking}~\ref{lem_catenoidSphereShrinking_deltaToZero}). Then, choose $ \lambda_1>\Lambda $ sufficiently large such that $ W_{\alpha/\beta}(\lambda,\lambda,\delta)< 4\pi + \eps $ for all $ \lambda>\lambda_1 $ (see Lemma~\ref{lem_catenoidSphereShrinking}~\ref{lem_catenoidSphereShrinking_lambdaToInfty}).
	Apply Proposition~\ref{prop_gluingA} for $ \delta $ and $ R=\lambda>\lambda_1 $. Let $ H $ be the obtained homotopy and denote $ g_0:=H_1 $ and $ g^i_0:=H_1|_{S_i} \in\CatSph(\lambda,\lambda,\delta) $, $ i=1,2 $. If $ g^1_0,g^2_0 $ are both catenoid spheres of type $ \alpha $, then $ g_0\in J $ and we are done. We may now assume that $ g^1_0 $ is a catenoid sphere of type $ \beta $. We deform $ g^1_0 $ by a homotopy of catenoid spheres $ g^1_t := \lambda_tf_{\lambda_t} $ with $ f_{\lambda_t}\in \CatSph_\beta(\lambda_t,\lambda_t,\delta) $ and
	$$\lambda_t = (1-t)\lambda + t\Lambda,\quad t\in[0,1]. $$
	By Lemma~\ref{lem_catenoidSphereShrinking} and $ \lambda>\Lambda $ and $\delta<\delta_0 $, this is a regular homotopy of decreasing Willmore energy to an immersion $ g^1_1\in\CatSph_\beta(\Lambda,\Lambda,\delta) $ of Willmore energy $$ W(g^1_1) = W_\beta(\Lambda,\Lambda,\delta)=4\pi - \eps. $$ This regular homotopy leaves the center piece of the catenoid fixed at all times and thus naturally extends to a regular homotopy $ g_t $ that is equal to $ g^1_t $ on $ S_1 $ and $ g^2_0 $ on $ S_2 $. We conclude 
	$$ W(g_1) = W(g^1_1)+W(g^2_0) < 4\pi - \eps + 4\pi + \eps = 8\pi.  $$ 
	Therefore, by Theorem~\ref{thm_KS8pi}, the Willmore flow with initial surface $ g_1 $ yields the last part of the desired homotopy to a round sphere.
\end{proof}
\begin{figure}[H]
	\centering
	\newcommand{\drawCatSphSingle}[1]{
		\begin{tikzpicture}[scale=0.65]
			\begin{axis}[
				scale=1,
				axis equal,
				axis lines=middle,
				xtick={-1,1},    
				ytick={0},    
				xmin=0.5,
				xmax=-0.5,
				ymin=-1.2,
				ymax=16.5,
				samples=500,
				legend pos=north west,
				xticklabel style={yshift=-5pt}
				]
				\pgfmathsetmacro{\lam}{#1}
				\pgfmathsetmacro{\R}{#1}
				\pgfmathsetmacro{\tzero}{ln(\lam + sqrt(\lam^2 - 1))/\lam}
				\pgfmathsetmacro{\theta}{rad(asin(1/\R))}
				
				\addplot[myred, thick, domain=0:\tzero,parametric] ({cosh(x*\lam)/\lam},{x-\tzero});
				\addplot[myred, thick, domain=0:\tzero,parametric] ({-cosh(x*\lam)/\lam},{x-\tzero});
				
				\addplot[myblue, thick, domain=\theta:2*pi-\theta,parametric] ({\R*sin(deg(x))},{\R*(cos(deg(\theta))-cos(deg(x)))});
				
			\end{axis}
		\end{tikzpicture}
	}
	\resizebox{0.9\textwidth}{!}{
	\begin{tabular}{c c c}
		\drawCatSphSingle{8}&
		\drawCatSphSingle{5}&
		\drawCatSphSingle{2}
	\end{tabular}
	}
	\caption{The homotopy $ f_{\lambda_t} $ from the proof of Theorem~\ref{thm_globalSing} for $ \lambda_t = 8,5,2 $.}\label{fig_catenoidSpheresShrinking}
\end{figure}
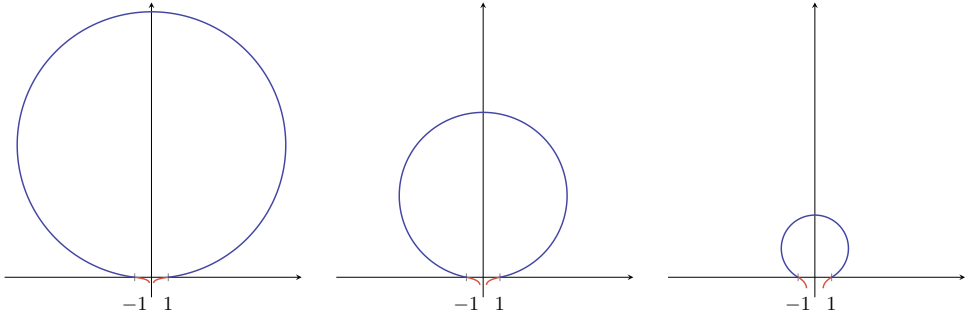

We conclude the section with some observations about the family $ J $ defined in Definition~\ref{def_J}.
\begin{lem}\label{lem_Jlem}
	Let $ \omega\in(8\pi,12\pi] $. There exists $ j\in J $ with $ W(j)<\omega $ such that for all $ f\in I(\omega)\, \cap\, J $, there exists a regular homotopy in $ I(\omega) $ from $ f $ to $ j $ or $ -j $.
\end{lem}
\begin{proof}
	Fix $ \Lambda>0 $ from Lemma~\ref{lem_catenoidSphereShrinking}~\ref{lem_catenoidSphereShrinking_lambdaToInfty}. Then, we choose $ \delta\in\left(0,\frac12\right) $ and $ \lambda_1 \geq \Lambda $ such that 
	\begin{equation}
		W(\lambda,\lambda,\delta)<\frac\omega2 \label{eq_WllSmall}
	\end{equation}
	for all $ \lambda \geq \lambda_1 $. Let $ j $ be a smoothly immersed sphere such that on each hemisphere $ S_i $ it is a catenoid sphere $ j|_{S_i}\in \CatSph_\alpha(\lambda_1,\lambda_1,\delta) $ with outer normal on the spherical parts. It has Willmore energy
	\begin{align*}
		W(j) = 2W_\alpha(\lambda_1,\lambda_1,\delta)<\omega.
	\end{align*}
	Now, for given $ f\in I(\omega)\, \cap\, J  $, we apply Proposition~\ref{prop_gluingA} for $ \delta $ and $ \lambda=R>\lambda_1 $ to obtain a regular homotopy to some $ i\in\Imm(\S^2,\R^3) $ with hemispheres $ i|_{S_i}\in \CatSph_\alpha(\lambda,\lambda,\delta) $. Type $ \beta $ cannot occur by Theorem~\ref{thm_triplePoints}. We concatenate the regular homotopy that deforms each hemisphere by $ t\mapsto i_t|_{S_i}\in \CatSph(\lambda -t,\lambda -t,\delta)  $, $ t\in[0,\lambda-\lambda_1] $. By \eqref{eq_WllSmall}, we have $ W(i_t)<\omega $ for all $ t\in[0,1] $. The obtained surface has either outer or inner normal on the spherical parts and thus can be joined to $ j $ or $ -j $ by reparametrizing homotopies.
\end{proof}
\begin{rem}\label{rem_JW>8pi}
	There exists a sequence of immersions $ j_n\in J $ with $ W(j_n)\to 8\pi $. This follows from Lemma~\ref{lem_Jlem} but also already from \cite[Section~5]{Blatt}. For $ j\in J $, we have $ W(j)\geq 8\pi $ by the Li--Yau inequality. We even have $ W(j)>8\pi $; otherwise, the Willmore flow would yield some $ j_t\in J $ with $ W(j_t)<8\pi $ since $ j $ is not a critical point.
	We conclude that the minimal Willmore energy $ \inf_{j\in J}W(j) = 8\pi $ is not assumed within the family $ J $.
\end{rem}
\subsection{Partial Global Classification of Singularities below $ 16\pi $}\label{sec_globalSing16pi}
If we replace the energy bound of $ 12\pi $ by $ 16\pi $, we also have to consider the possibility of catenoid, trinoid and Enneper singularities \cite{LN}. 
Here, \textit{trinoid} refers to any complete genus zero minimal surface of total curvature $ -8\pi $ with three ends, as classified in \cite[Theorem~3]{Lopez_CompleteMSTotCurv}\com{In (46) second formula it should be $ [\dots] + r_1^2 =0 $ and not $ [\dots] - r_1^2 =0 $}. These have either two catenoid ends and one planar end or three catenoid ends which can have varying orientations, neck sizes and axes (with some restrictions, cf.\ \cite{K_Bubbles,Perez_RelMS}).

The approach of the previous section capitalizes on the fact that the minimal surface that models the singularity---the catenoid---has more than one end. This allowed us to divide the remaining surface around the blow-up region into parts of smaller Willmore energy. This can be applied to trinoids as well.
Enneper's surface, however, has only one end and would require an entirely different method and is therefore excluded in the following partial classification.

The construction of catenoid spheres in Section~\ref{sec_gluingCatSph} can be generalized to include not only catenoids but also asymptotic catenoid ends and planar ends. Hence, planar ends can be connected to a round sphere, and catenoid ends can have a connection to a round sphere of type $ \alpha $ or $ \beta $. 
This leads to the following additional cases.
\begin{prop}\label{prop_globalSing16pi}
	Let $ f:\S^2 \to \R^3 $ be a smooth immersion with Willmore energy less than $ 16\pi $. If $ f $ does not develop a singularity of Enneper type (see Theorem~\ref{thm_LN}), then there exists a regular homotopy $ H:\S^2\times [0,1]\to\R^3 $ with $ H_0=f $ and $ W(H_t)<W(f) $ for all $ t\in(0,1] $, where $ H_1 $ realizes one of the following models (see also Figure~\ref{fig_globalSing16pi}).
	\begin{enumerate}[label=(\alph*)]
		\item A round sphere.\label{enum_16pi_sphere}
		\item \label{enum_16pi_J} A catenoid whose ends are $ \alpha $-connected to round spheres (see Proposition~\ref{prop_gluingA}).
		
		\item \label{enum_16pi_DoubleJ} Two catenoids whose four ends are each $ \alpha $-connected to round spheres.
		
		\item \label{enum_16pi_trinoid} A trinoid whose ends are connected to round spheres and whose connections at the catenoid ends are of type $ \alpha $.
		
	\end{enumerate}
\end{prop}
\begin{figure}[htbp]
	\centering
	\subfigure[]{\includegraphics[width=0.2\textwidth]{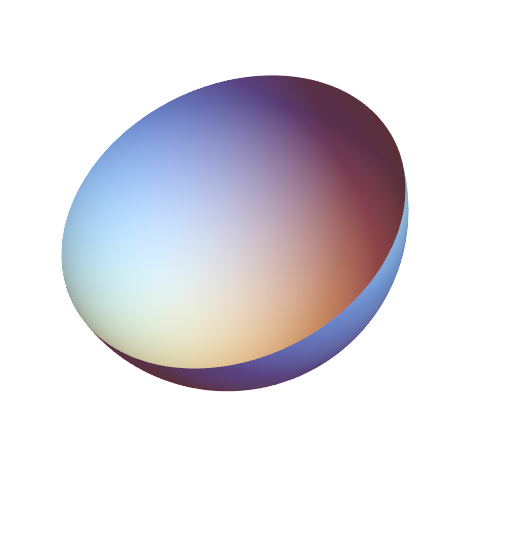}}
	\hspace*{-0.75cm}
	\subfigure[]{\includegraphics[width=0.2\textwidth]{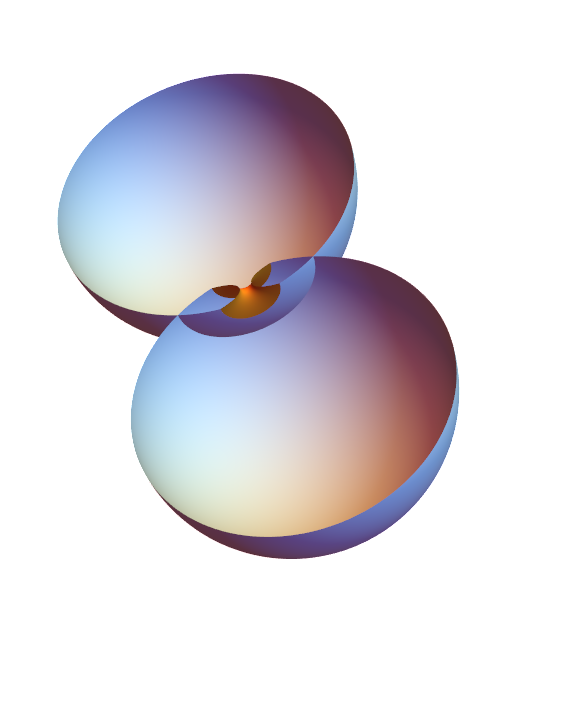}}
	\hspace*{-0.75cm}
	\subfigure[]{\includegraphics[width=0.2\textwidth]{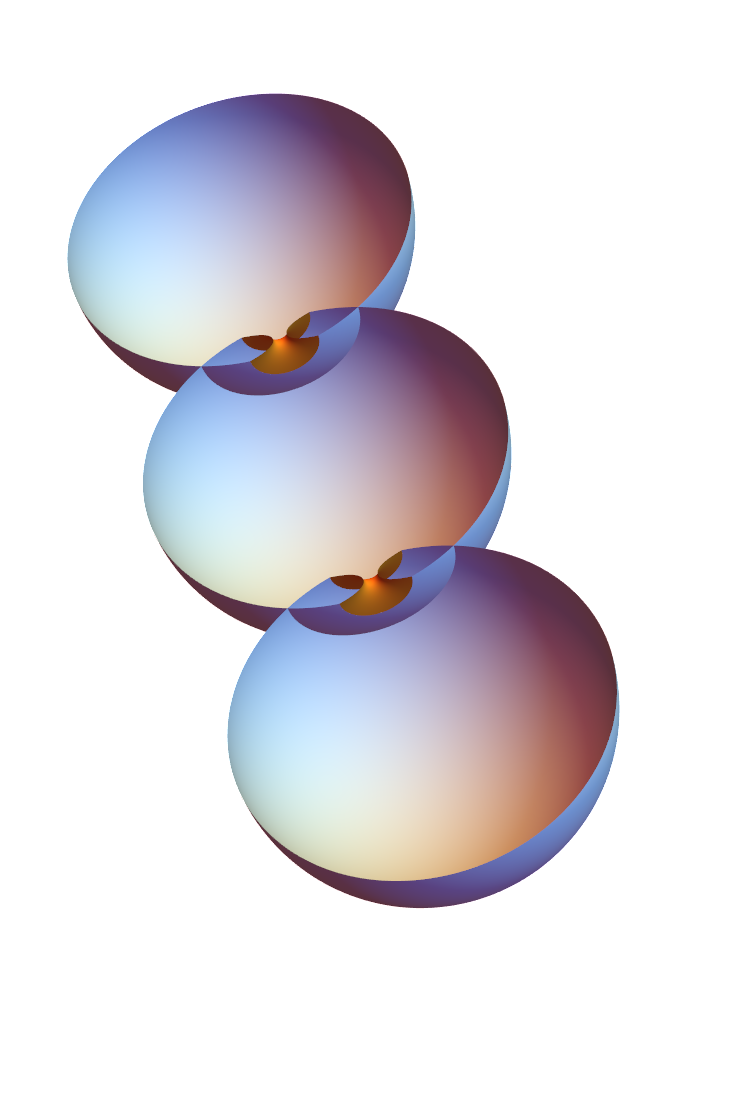}}
	\hspace*{-0.75cm}
	\subfigure[]{
		\raisebox{0cm}{\includegraphics[width=0.25\textwidth]{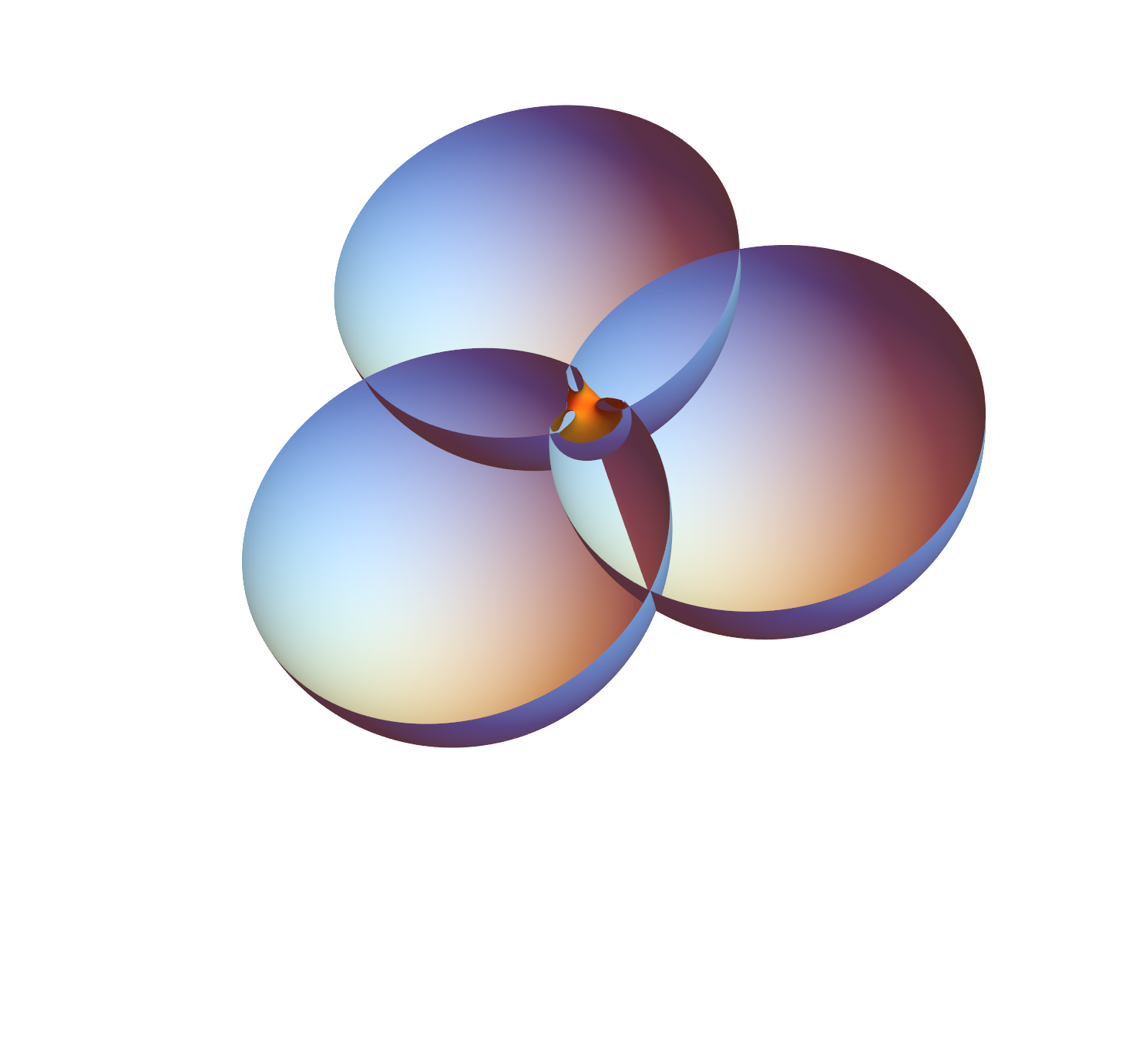}}
		\hspace*{-1.7cm}
		\raisebox{1.5cm}{\includegraphics[width=0.23\textwidth]{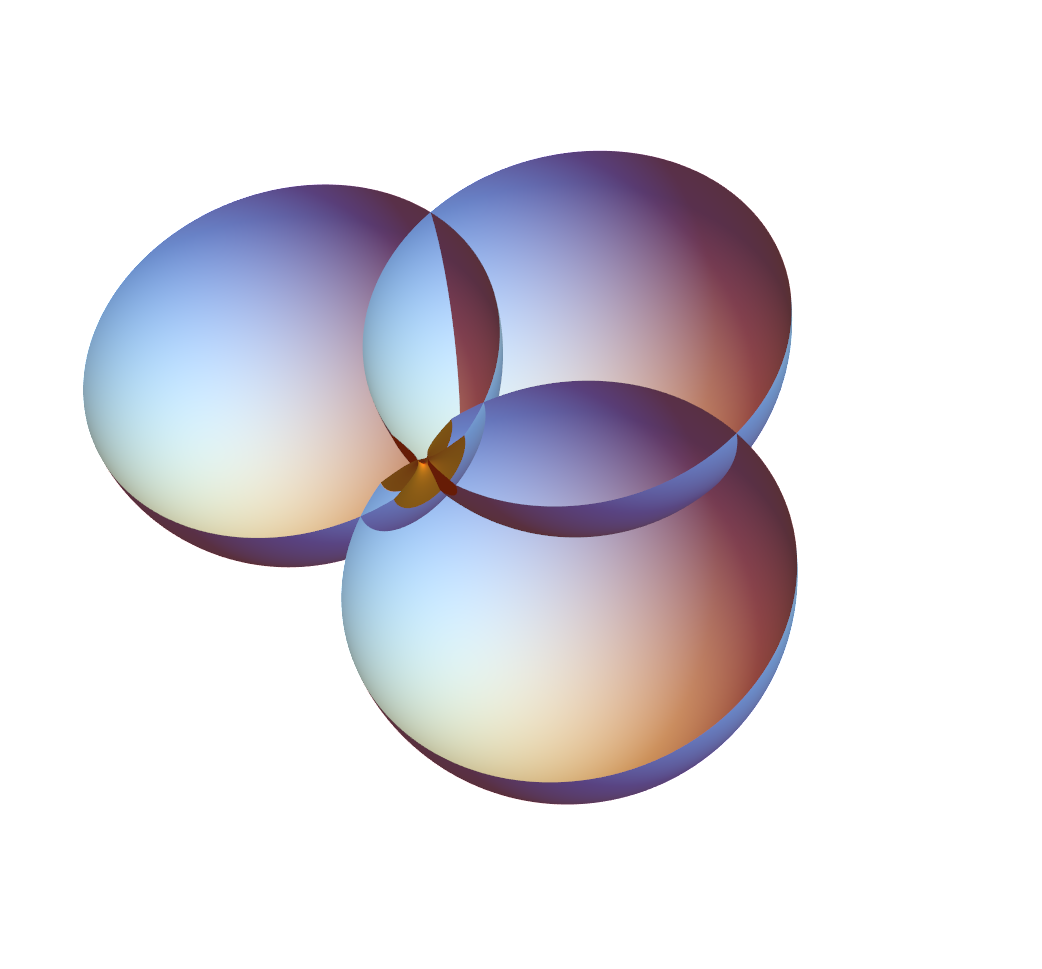}}
}
	\caption{The singularity models of Proposition~\ref{prop_globalSing16pi}. Their images in the lower half-space are displayed, and a reflection at the horizontal plane gives their full images.}\label{fig_globalSing16pi}
\end{figure}
\begin{proof}
	We repeat many arguments of the proof of Proposition~\ref{prop_gluingA} and Theorem~\ref{thm_globalSing} without going into detail. 
	If the Willmore flow with initial surface $ f $ does not converge to a round sphere, it develops a singularity whose blow-up is either a catenoid, a trinoid, or Enneper's surface (see Theorem~\ref{thm_LN}). The latter was ruled out by assumption.\\
	Trinoid case:\\
	Analogously to the catenoid case in Proposition~\ref{prop_gluingA}, each of the three disks that close up the trinoid has Willmore energy less than $ 8\pi $, and with the same gluing procedure, we arrive at a trinoid whose ends are $ \alpha $-connected or $ \beta $-connected to round spheres. Note that Lemma~\ref{lem_catenoidSphereShrinking} uses the asymptotics of the catenoid and thus equally applies to catenoid ends of the trinoid. So, if one of the connections at a catenoid end is of type $ \beta $, we employ Lemma~\ref{lem_catenoidSphereShrinking} as in the proof of Theorem~\ref{thm_globalSing} to arrive at an immersion of Willmore energy less than $ 12\pi $ from which we apply Proposition~\ref{prop_gluingA} to obtain the case \ref{enum_16pi_sphere} or \ref{enum_16pi_J}. But if all are of type $ \alpha $, we arrive at \ref{enum_16pi_trinoid}.\\
	Catenoid case:\\
	The disks that close up a catenoid singularity now each have Willmore energy less than $ 12\pi $, but only one of them can have Willmore energy larger than $ 8\pi $. For this part, we proceed as in the proof of Proposition~\ref{prop_gluingA}, but instead of using the Willmore flow to a round sphere to construct the homotopy $ H^1 $, we use the regular homotopy obtained from Proposition~\ref{prop_gluingA} to one of the cases \ref{enum_16pi_sphere} or \ref{enum_16pi_J}. In the former case, the remaining proof carries through to yield \ref{enum_16pi_sphere} or \ref{enum_16pi_J} for the whole surface, whereas in the latter case we obtain a catenoid piece that is attached to a sphere and to a surface of type \ref{enum_16pi_J}. Both connections can be of either type $ \alpha $ or type $ \beta $. If one of them is of type $ \beta $, we employ Lemma~\ref{lem_catenoidSphereShrinking} to get below an energy of $ 12\pi $ and apply Proposition~\ref{prop_gluingA} again. But if both connections are of type $ \alpha $, we arrive at a surface described in \ref{enum_16pi_DoubleJ}. 
\end{proof}

We outline a potential notion of ``Willmore flow with surgery'' for spheres with energy at most $ 16\pi $.
\begin{rem}\label{rem_surgery}
	Whenever the Willmore flow with initial energy at most $ 12\pi $ does not converge, we replace the catenoid by two flat disks. The convergence of the two resulting Willmore flows of spheres is automatically obtained by Theorem~\ref{thm_KS8pi} \cite{KS_Remov} for initial energies at most $ 8\pi $ (see the proof of Proposition~\ref{prop_gluingA}). Hence, with this ``surgery'' we obtain convergence for all initial surfaces with energy at most $ 12\pi $ to either \textit{one} or \textit{two} round spheres. However, it is not clear whether this is a ``canonical'' surgery or whether there exist other ways of flowing ``through'' a catenoid singularity with just one sphere. Compare this to the setting of tori of revolution, where the authors of \cite{DMRS_WillmoreFlowToriFixedConformal} propose a way of flowing ``through'' singularities of certain Willmore flows with a change of topology from torus to sphere.\com{If the initial surface has turning number zero and mirror symmetry and the flow is immortal and has bounded area, then they show that it converges to an inverted catenoid. They provide a spherical Willmore flow that admits the (nonsmooth) inverted catenoid as initial surface}
	The discussion above for Willmore energy at most $ 16\pi $ yields the following.
	
	\textit{For smooth immersions $ f_0:\S^2\to\R^3 $ with $ W(f_0)\leq 16\pi $, a notion of ``Willmore flow with surgery'' with initial surface $ f_0 $ either has an Enneper-type singularity or converges to one, two or three round spheres.}
\end{rem}
\appendix
\section{Appendix}\label{sec_appendixGlue}
\corTriplePoints*
\begin{proof}
	By the Li--Yau inequality, it follows immediately that there exists a time $ t\in[0,1] $ with $ W(H_t)\geq 12\pi $. But actually a strict inequality holds, since \cite{S_TriplePoints} implies that there must be an immersion $ H_t $ with a generic triple point; in particular, all immersions in a $ C^\infty $-neighborhood of that immersion have triple points\com{Otherwise a small perturbation would yield a regular homotopy without triple points}. Thus, any Willmore flow started in $ H_t $ must still have energy at least $ 12\pi $ after some time $ \eps>0 $. And since there are no critical points at $ 12\pi $, it follows that $ W(H_t)>12\pi $.
\end{proof}

\propLiYauTurningNumber*
\begin{proof}
	First, define for $ r_0>0 $, $ h_0\in\R $ the quarter circles 
	\begin{align*}
		\gamma^I_{r_0,h_0}&:\left[-\frac\pi2,0\right]\to[0,\infty)\times \R,\qquad &t\mapsto  \left(r_0\cos t, r_0\sin t + h_0\right),\\
		\gamma^{II}_{r_0,h_0}&:\left[0,\frac\pi2\right]\to[0,\infty)\times \R,\qquad &t\mapsto  \left(r_0\cos t, r_0\sin t + h_0\right).
	\end{align*}
	Without loss of generality, we may assume $ \tau>0 $. Let $ \vartheta:[0,1]\to\R $ be the lift of the tangent angle of $ c $ with $ \vartheta(0)=0 $ and $$ \vartheta(1)=2\pi\tau = 2\pi k + \pi $$ for some $ k\in \N $. Let $ l\in\{0,1,\dots,k\} $ and denote $ c=(r,h) $. The mean value theorem lets us pick $ t_l $ to be the smallest number in $ \vartheta^{-1}(\{2\pi l + \frac\pi2 \}) $. The tangent at $ t_l $ is $ (0,1) $. In the following, we construct $ C^1 $-curves that are composed of smooth curves connected to each other at points where their tangent is $ (0,1) $. For $ l\in\{0,1,\dots,k-1\} $, consider the $ C^1 $-curve $ \gamma_l $ obtained as the concatenation of 
	\begin{align*}
		\gamma^I_{r(t_l),h(t_l)},\quad c|_{[t_l,t_{l+1}]},\quad \gamma^{II}_{r(t_{l+1}),h(t_{l+1})}
	\end{align*}
	Now, the turning number of $ \gamma_l $ is $ \frac32 $. In an abuse of notation, we identify each curve segment $ \rho $ with the surface obtained by rotating the curve and write $ W(\rho) $ for its Willmore energy. Furthermore, we can approximate piecewise smooth curves by smooth curves. They can be chosen to have Willmore energy arbitrarily close to the sum of the energies of each smooth piece. We claim
	\begin{align}\label{eq_turningNumber8pi}
		W(\gamma_l)\geq 8\pi.
	\end{align}
	This can be deduced in two ways.
	First, this is due to the fact that the closed plane curve obtained from mirroring $ \gamma $ has turning number $ 3 $ and thus by Hopf's Umlaufsatz (or Theorem of turning tangents \cite{DC}) has self-intersections. Therefore, the immersion $ f $ is not embedded and by the Li--Yau inequality, the claim follows. Second, the surface of revolution and the turning number are preserved under the Willmore flow \cite{Blatt}; hence, the Willmore flow with initial surface $ f $ cannot converge to a round sphere and by Theorem~\ref{thm_KS8pi}, the claim follows. The two quarter circles produce two round hemispheres and each contributes $ 2\pi $ energy to \eqref{eq_turningNumber8pi}. Hence, we obtain
	$$ W(c|_{[t_l,t_{l+1}]})\geq 4\pi. $$
	Furthermore, we construct two $ C^1 $-spheres of revolution by concatenating
	\begin{align*}
		c|_{[0,t_0]}, \quad \gamma^{II}_{r(t_0),h(t_0)},
	\end{align*}
	and
	\begin{align*}
		\gamma^I_{r(t_k),h(t_k)},\quad c|_{[t_k,1]}.
	\end{align*}
	Both have Willmore energy at least $ 4\pi $, where $ 2\pi $ is contributed by the round hemispheres. Therefore,
	\begin{align*}
		W(c|_{[0,t_{0}]}),W(c|_{[t_k,1]}) \geq 2\pi,
	\end{align*}
	and we obtain in total (see Figure~\ref{fig_LiYauTurningNumber} for an example)
	\begin{align*}
		W(f)&\geq W(c|_{[0,t_{0}]}) + \sum_{l=0}^{k-1}W(c|_{[t_l,t_{l+1}]}) + W(c|_{[t_k,1]})\\
		&\geq2\pi + 4\pi k + 2\pi\\
		&= 4\pi \left(\tau +\frac12\right).
	\end{align*}
	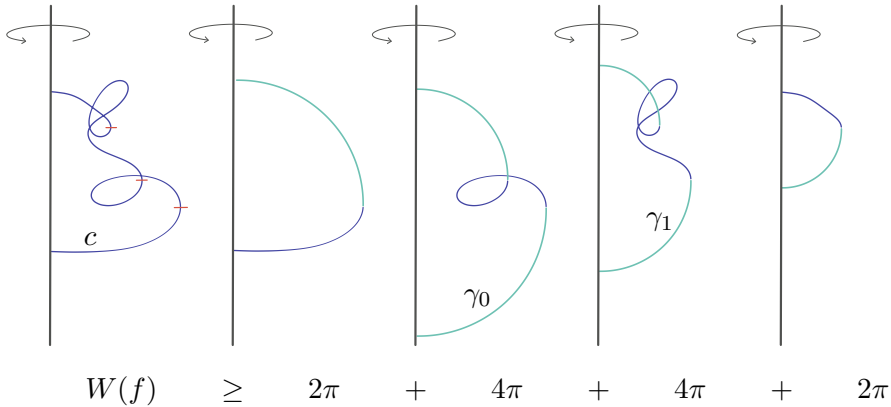
\begin{figure}[H]
		\centering
		\begin{tikzpicture}
			\def\d{\textwidth*0.05}
			\node[anchor=south] at (0,0) {\includegraphics[width=0.75\textwidth]{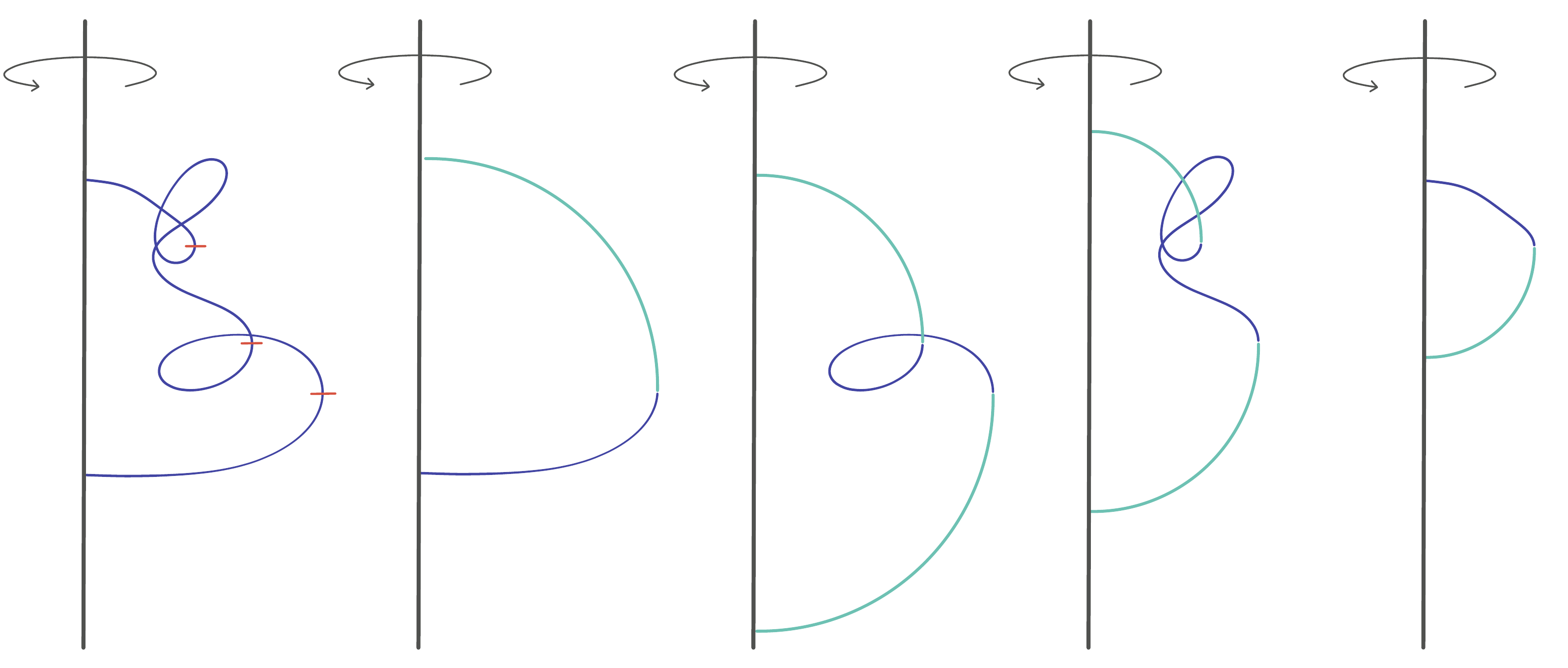}};
			\node[anchor=center] at (-6\d,2.25\d) {$ c $};
			\node[anchor=center] at (0.8\d,1.2\d) {$ \gamma_0 $};
			\node[anchor=center] at (4\d,2.55\d) {$ \gamma_1 $};
			\def\x{0.039\textwidth}
			\node[anchor=north] at (0.73,0) {$ W(f) $ \hspace{\x} $ \geq $ \hspace{\x} $ 2\pi $ \hspace{\x} $ + $ \hspace{\x} $ 4\pi $ \hspace{\x} $ + $ \hspace{\x} $ 4\pi $ \hspace{\x} $ + $ \hspace{\x} $ 2\pi $};
		\end{tikzpicture}
		\caption{An example for the proof of Proposition~\ref{prop_LiYauExtensionRotSymm}.}\label{fig_LiYauTurningNumber}
	\end{figure}
	It remains to show the strict inequality. Assume it is an equality; then $ f $ is a critical point of the Willmore energy. Otherwise, the Willmore flow would yield a surface contradicting the inequality since it preserves surfaces of revolution and their turning number. By assumption, $ f $ is not a round sphere. By Bryant's classification \cite{Bryant_Dual}, with $ n:=\left(\tau + \frac12\right)\geq 2 $, the immersion $ f $ has an $ n $-tuple point at which it can be inverted to obtain a complete minimal surface with $ n $ embedded planar ends. But, as a surface of revolution, $ f $ has a full curve of $ n $-tuple points and the minimal surface also has $ n $-tuple points as well as density $ n $ at infinity. This is a contradiction to \cite[Theorem~A]{Kusner_Conformal}.
\end{proof}

Here, we present the computations of the Willmore energy of catenoid spheres omitted in Lemma~\ref{lem_catenoidSphereShrinking}.
\begin{lem}\label{lem_catenoidSpheresShrinkingAppendix}
	There exist constants $ \Lambda,C>0 $ such that for all $ \lambda\geq \Lambda $ and $ \delta\in\left(0,\frac12\right) $
	\begin{enumerate}[label=\roman*)]
		\item $ W^{glue}_{\alpha}(\lambda,\delta) \leq C\frac{1}{\delta\lambda^2}, $ \label{lem_catenoidSpheresShrinkingAppendix_alpha}
		\item $ W^{glue}_{\beta}(\lambda,\delta) \leq C \frac{\delta}{\lambda^2}, $ \label{lem_catenoidSpheresShrinkingAppendix_beta}
		\item $ \partial_\lambda W^{glue}_{\beta}(\lambda,\delta) \geq -  C \frac{\delta}{\lambda^3}. $ \label{lem_catenoidSpheresShrinkingAppendix_monotonicity}
	\end{enumerate}
\end{lem}
\begin{proof}
	For a general surface of revolution $ f(t,\varphi) = (r(t)\cos(\varphi),r(t)\sin(\varphi),h(t)) $ where we set $ r(t)=t $ and $ h:I\to\R $, we obtain formulas for its mean curvature $ H $, the mean of the principal curvatures $ \kappa_1 $ and $ \kappa_2 $, and its first fundamental form $ G $. With $ l(t):=\sqrt{1+h'(t)^2} $, we have
	\begin{align*}
		H(t,\varphi) &= \frac12 (\kappa_1(t,\varphi) + \kappa_2(t,\varphi)) \\
		&= \frac12 \left(\frac{h''(t)}{l(t)^3} + \frac{h'(t)}{t l(t)}\right)\\
		&= \frac{t h''(t) + l(t)^2 h'(t)}{2t l(t)^3},\\
		\sqrt{\det G(t,\varphi)}&= \sqrt{l(t)^2 t^2} = t l(t),\\
		H(t,\varphi)^2 \sqrt{\det G(t,\varphi)} &= \frac{(th''(t) + l(t)^2 h'(t))^2}{4tl(t)^5} = \frac{\left( th''(t) + h'(t) + h'(t)^3 \right)^2}{4 t (1+h'(t)^2)^\frac52}.
	\end{align*}
	We first lay out the details of a catenoid sphere of type $ \beta $ and omit the index $ \beta $. Let $ \lambda>1 $ and $ \delta\in\left(0,\frac12\right) $. Let $ u:(1/\lambda,\infty)\to\R $ and $ v:[0,R)\to\R $ be the graph parametrizations of the generating curves $ (t,u(t)) $ of the catenoid and $ (t,v(t)) $ of the sphere in the catenoid sphere construction, respectively (see Definition~\ref{def_CatSph}). We have the catenoid
	\begin{align*}
		u(t):&=\frac{\arccosh(\lambda t) - \arccosh(\lambda)}{\lambda},\\
		u'(t) &= \frac{1}{\sqrt{\lambda^2 t^2 -1}},\\
		u''(t) &= -\frac{\lambda^2 t}{(\lambda^2 t^2 -1)^{\frac32}},\\
		u'''(t) &= \frac{\lambda^2 (2\lambda^2 t^2 +1)}{(\lambda^2 t^2 -1)^{\frac52}},\\
		\partial_\lambda u(t) &= \lambda^{-2}\left( \arccosh(\lambda) - \arccosh(\lambda t) + \lambda t \left(\lambda^2 t^2 -1 \right)^{-\frac12} - \lambda \left(\lambda^2 -1\right)^{-\frac12} \right),\\
		\partial_\lambda u'(t) &= -\frac{\lambda t^2}{(\lambda^2 t^2 -1)^{\frac32}},\\
		\partial_\lambda u''(t) &= \frac{\lambda t (2+ \lambda^2 t^2)}{(\lambda^2 t^2 -1)^{\frac52}},\\
		\partial_\lambda u'''(t) &= -\frac{\lambda (2\lambda^4 t^4 + 11\lambda^2 t^2 +2)}{(\lambda^2 t^2 -1)^{\frac72}},
	\end{align*}\allowdisplaybreaks
	and the sphere 
	\begin{align*}
		v(t):&= \sqrt{\lambda^2 - 1} - \sqrt{\lambda^2-t^2},\\
		v'(t)&= \frac{t}{\sqrt{\lambda^2-t^2}},\\
		v''(t)&= \frac{\lambda^2}{(\lambda^2-t^2)^{\frac32}},\\
		v'''(t)&= \frac{3\lambda^2 t}{(\lambda^2-t^2)^{\frac{5}{2}}},\\
		\partial_\lambda v(t) &=\frac{\lambda}{\sqrt{\lambda^2-1}} - \frac{\lambda}{\sqrt{\lambda^2-t^2}},\\
		\partial_\lambda v'(t) &= -\frac{\lambda t}{(\lambda^2-t^2)^{\frac32}},\\
		\partial_\lambda v''(t) &= -\frac{\lambda^3 +2 \lambda t^2}{(\lambda^2-t^2)^{\frac52}},\\
		\partial_\lambda v'''(t) &= -\frac{3\lambda t (3\lambda^2 + 2 t^2)}{(\lambda^2-t^2)^{\frac72}}.
	\end{align*}
	In the gluing region $ I:=[1-\delta,1+\delta] $, the generating curve $ h:I\to\R^2 $ that interpolates between the catenoid and the sphere is defined and satisfies the following.
	\begin{align*}
		h:&=u + \phi (v-u),\\ 
		h'&=u' + \phi (v'-u') + \phi' (v-u),\\
		h''&= u'' + \phi(v''-u'') + 2\phi' (v'-u') + \phi'' (v-u),\\
		\partial_\lambda h'&= \partial_\lambda u' + \phi \partial_\lambda (v'-u') + \phi' \partial_\lambda(v-u),\\
		\partial_\lambda h''&= \partial_\lambda u'' + \phi \partial_\lambda(v''-u'') + 2\phi' \partial_\lambda(v'-u') + \phi'' \partial_\lambda(v-u).
	\end{align*}
	We note that
	\begin{align}\label{eq_hTangentDec}
		\partial_\lambda h' = (1-\phi) \underset{<0}{\underbrace{\partial_\lambda u'}} + \phi \underset{<0}{\underbrace{\partial_\lambda v'}} + \phi' \underset{\leq 0}{\underbrace{\partial_\lambda \left( v - u \right)}} <0,
	\end{align}
	where the last term $ w(t):=\partial_\lambda (v(t)-u(t)) $ is nonpositive due to the fact that $ w(1)=0 $, $ w'(1)=0 $ and $ w''<0 $, which can be directly deduced from the explicit formulas.
	Furthermore, we can infer the asymptotic behavior of the explicit functions by elementary estimations, for example, we get for all $ \lambda\geq 3 $
	\begin{align}\label{eq_WglueElementaryEst}
		\abs{v''(t)}\leq \frac{\lambda^2}{\left(\lambda^2-2^2\right)^{\frac32}} \leq \frac{\lambda^2}{\left( \lambda^2/2 \right)^{\frac32}} \leq \frac{8}{\lambda}.
	\end{align}
	We gather some estimations. In the following, we write $ A(\lambda,\delta,t)\lesssim B(\lambda,\delta,t) $ if there exist constants $ \Lambda>0 $ and $ C>0 $ independent of $ \lambda $, $ \delta $ and $ t $ such that $ A(\lambda,\delta,t)\leq C B(\lambda,\delta,t) $ for all $ \lambda\geq\Lambda $, $ \delta\in(0,\frac12) $, $ t\in [1-\delta,1+\delta] $. 
	Using Taylor expansion at $ t=1 $ with the remainder term in Lagrange form and elementary estimations as in \eqref{eq_WglueElementaryEst}, we obtain the following.
	\begin{align*}
		\abs{v''(t)-u''(t)} &\leq \abs{v''(1) - u''(1)} + \delta \max_{\xi\in I} \abs{v'''(\xi)-u'''(\xi)}\lesssim \frac1\lambda + \frac\delta\lambda \\
		&\lesssim \frac1\lambda,\\
		\abs{v'(t) - u'(t)} &\leq \underset{=0}{\underbrace{\abs{v'(1) - u'(1)}}} + \delta \max_{\xi\in I} \abs{v''(\xi)-u''(\xi)}\\
		&\lesssim \frac\delta\lambda,\\
		\abs{v(t) - u(t)} &\leq \underset{=0}{\underbrace{\abs{v(1) - u(1)}}} + \delta \underset{=0}{\underbrace{\abs{v'(1) - u'(1)}}} + \frac{\delta^2}{2} \max_{\xi\in I} \abs{v''(\xi)-u''(\xi)} \\ 
		&\lesssim \frac{\delta^2}{\lambda},\\
		\abs{\partial_\lambda v''(t)- \partial_\lambda u''(t)} &\leq \abs{\partial_\lambda v''(1) - \partial_\lambda u''(1)} + \delta \max_{\xi\in I} \abs{\partial_\lambda v'''(\xi)- \partial_\lambda u'''(\xi)}\\
		&\lesssim \frac{1}{\lambda^2} + \frac{\delta}{\lambda^2}
		\lesssim \frac{1}{\lambda^2},\\
		\abs{\partial_\lambda v'(t) - \partial_\lambda u'(t)} &\leq \underset{=0}{\underbrace{\abs{\partial_\lambda v'(1) - \partial_\lambda u'(1)}}} + \delta \max_{\xi\in I} \abs{\partial_\lambda v''(\xi)- \partial_\lambda u''(\xi)}\\
		&\lesssim \frac{\delta}{\lambda^2},\\
		\abs{\partial_\lambda v(t) - \partial_\lambda u(t)} &\leq \underset{=0}{\underbrace{\abs{\partial_\lambda v(1) - \partial_\lambda u(1)}}} + \delta \underset{=0}{\underbrace{\abs{\partial_\lambda v'(1) - \partial_\lambda u'(1)}}}\\ &\quad + \frac{\delta^2}{2} \max_{\xi\in I} \abs{\partial_\lambda v''(\xi)- \partial_\lambda u''(\xi)}\\ 
		&\lesssim \frac{\delta^2}{\lambda^2}.
	\end{align*}
	We obtain for the interpolating function using $ \abs{\phi}\leq 1 $, $ \abs{\phi'}\lesssim \frac{1}{\delta} $ and $ \abs{\phi''}\lesssim \frac{1}{\delta^2}, $
	\begin{align*}
		\abs{h'(t)} &\lesssim \abs{u'(t)} + \abs{v'(t)-u'(t)} + \frac1\delta \abs{v(t)-u(t)} \lesssim \frac1\lambda + \frac1\lambda + \frac{\delta}{\lambda} \lesssim \frac1\lambda,\\
		\abs{h''(t)} &\lesssim \abs{u''(t)} + \abs{v''(t)-u''(t)} + \frac1\delta \abs{v'(t)-u'(t)} + \frac{1}{\delta^2} \abs{v(t)-u(t)} \lesssim \frac1\lambda,\\
		\abs{\partial_\lambda h'(t)} &\lesssim \abs{\partial_\lambda u'(t)} + \abs{\partial_\lambda v'(t)-\partial_\lambda u'(t)} + \frac1\delta \abs{\partial_\lambda v(t)-\partial_\lambda u(t)}\lesssim \frac{1}{\lambda^2},\\
		\abs{\partial_\lambda h''(t)} &\lesssim \abs{\partial_\lambda u''(t)} + \abs{\partial_\lambda v''(t)-\partial_\lambda u''(t)} + \frac1\delta \abs{\partial_\lambda v'(t)-\partial_\lambda u'(t)}\\ 
		\phantom{\abs{\partial_\lambda h''(t)}}&\phantom{\lesssim \abs{\partial_\lambda u''(t)}} + \frac{1}{\delta^2} \abs{\partial_\lambda v(t)-\partial_\lambda u(t)} \lesssim \frac{1}{\lambda^2}.
	\end{align*}
	We plug the above estimations into the formula for the gluing energy
	\begin{align}
		W^{glue}(\lambda,\delta) &= \int_{1-\delta}^{1+\delta} \int_{0}^{2\pi} H(t,\varphi)^2 \sqrt{\det G(t,\varphi)}\ d\varphi dt \notag\\
		&=2\pi \int_{1-\delta}^{1+\delta} \frac{\left( th''(t) + h'(t) + h'(t)^3 \right)^2}{4 t (1+h'(t)^2)^\frac52}\ dt \label{eq_Wglue}\\
		&\leq 2\pi \delta \max_{t\in I}  \left( th''(t) + h'(t) + h'(t)^3 \right)^2,\notag
	\end{align}
	and obtain
	\begin{align}
		W^{glue}_\beta(\lambda,\delta) & \lesssim \delta \left(\frac{1}{\lambda}+ \frac{1}{\lambda} + \frac{1}{\lambda^3}\right)^2 \lesssim \frac{\delta}{\lambda^2},\label{eq_WglueEstimated}
	\end{align}
	which is \ref{lem_catenoidSpheresShrinkingAppendix_beta}.
	
	To prove \ref{lem_catenoidSpheresShrinkingAppendix_alpha} for type $ \alpha $, we note the few differences from the computations above. Let $ u_\alpha,v_\alpha,h_\alpha $ denote the corresponding catenoid, sphere and interpolating curve of the catenoid sphere of type $ \alpha $. We have $ u_\alpha=u $ and $ v_\alpha=-v $. The only situation where the opposite sign of $ v $ weakens the estimations is where $ v'(1)-u'(1)=0 $ for type $ \beta $ is used. This now becomes
	\begin{align*}
		\abs{v_\alpha'(1)-u_\alpha'(1)}\lesssim \frac1\lambda,
	\end{align*}
	which weakens the following estimations to
	\begin{align*}
		\abs{v_\alpha'-u_\alpha'}&\lesssim \frac1\lambda,\\
		\abs{v_\alpha-u_\alpha}&\lesssim \frac{\delta}{\lambda},\\
		\abs{h'_\alpha}&\lesssim \frac1\lambda,\\
		\abs{h''_\alpha}&\lesssim \frac{1}{\delta \lambda},
	\end{align*}
	and instead of \eqref{eq_WglueEstimated}, we obtain
	$$ W_\alpha^{glue}(\lambda,\delta)\lesssim \delta \left( \frac{1}{\delta \lambda} + \frac1\lambda + \frac{1}{\lambda^3} \right)^2 \lesssim \frac{1}{\delta \lambda^2}, $$
	which is \ref{lem_catenoidSpheresShrinkingAppendix_alpha}.
	
	It remains to show \ref{lem_catenoidSpheresShrinkingAppendix_monotonicity}.
	Writing $$ J(t):=\left( th''(t) + h'(t) + h'(t)^3 \right)^2, $$ 
	and using \eqref{eq_Wglue}, we obtain
	\begin{align*}
		\partial_\lambda W^{glue}(\lambda,\delta)&= \frac{\pi}{2} \int_{1-\delta}^{1+\delta} \partial_\lambda \left( \frac{J(t)}{t (1+h'(t)^2)^\frac52} \right) dt.
	\end{align*}
	We may use the fact that for functions $ f:\R\to [0,\infty) $, $ g:\R\to (0,\infty) $ with $ g'\leq 0 $, we have
	$$ \left(\frac{f}{g}\right)'=\frac{f'g - fg'}{g^2}\geq \frac{f'}{g}. $$ We have $ \partial_\lambda h' <0 $ by \eqref{eq_hTangentDec}, which implies $ \partial_\lambda \left( t \left(1+h'(t)^2\right)^{\frac52} \right) < 0 $ and thus
	\begin{align*}
		\partial_\lambda W^{glue}(\lambda,\delta) &\geq \frac{\pi}{2} \int_{1-\delta}^{1+\delta} \frac{\partial_\lambda J(t)}{ t (1+h'(t)^2)^\frac52}dt.
	\end{align*}
	A case distinction on the sign of $ \partial_\lambda J(t) $ implies with $ t\geq \frac12 $,
	\begin{align*}
		\frac{\partial_\lambda J(t)}{ t (1+h'(t)^2)^\frac52} \geq \min\left\{ 0 , 2 \partial_\lambda J(t) \right\}.
	\end{align*}
	We estimate further
	\begin{align*}
		\abs{\partial_\lambda J(t)}=&2\abs{ \left( th''(t) + h'(t) + h'(t)^3 \right) \left( t \partial_\lambda h''(t) + \partial_\lambda h'(t) + 3h'(t)^2\partial_\lambda h'(t) \right) }\\
		&\lesssim \left( \frac1\lambda + \frac{1}{\lambda} + \frac{1}{\lambda^3} \right) \left( \frac{1}{\lambda^2} + \frac{1}{\lambda^2} + \frac{1}{\lambda^4} \right) \lesssim  \frac{1}{\lambda^3},
	\end{align*}
	and obtain
	\begin{align*}
		\partial_\lambda W^{glue}(\lambda,\delta) &\gtrsim \int_{1-\delta}^{1+\delta} \min\left\{ 0, \partial_\lambda J(t) \right\} dt\\
		&\gtrsim \int_{1-\delta}^{1+\delta} -\frac{1}{\lambda^3} dt \\
		&\gtrsim - \frac{\delta}{\lambda^3},
	\end{align*}
	which is \ref{lem_catenoidSpheresShrinkingAppendix_monotonicity}.
\end{proof}

\begin{lem}\label{lem_uniformGraphRadius}
	For all $ A_{\max}>0 $, there exists a radius $ r(A_{\max})>0 $ such that all smooth immersions $ f:\Sigma\to\R^3 $ of a closed surface $ \Sigma $ with second fundamental form $$ \max_{p\in\Sigma}{\norm{A(p)}}\leq A_{\max} $$ are graphs over a ball $ B_r(0)\subseteq T_pf$ \com{$ B_r(0) $ w.r.t $ \norm{\cdot}_{g(p)} $}in the tangent space for all $ p\in\Sigma $ and their graph parametrizations $ u:B_r(0)\to \R $ satisfy 
	\begin{align*}
		\norm{D^2 u}_{C^0(B_r(0))}\leq 5 A_{\max}.
	\end{align*}
\end{lem}
\begin{proof}
	By \cite[Theorem~2.4]{Langer_Comp}, for all $ \alpha\in(0,1) $, there exists a radius $ r(\alpha)>0 $ such that for all immersions $ f:\Sigma\to\R^3 $ with $ \norm{A}\leq A_{\max} $ and for all $ p\in\Sigma $, the graph parametrization $ u:B_r(0)\to\R^3 $ satisfies
	\begin{align*}
	 	\norm{Du}_{C^0(B_r(0))}\leq \alpha.
	\end{align*}
	Together with (see, e.g., \cite[Proof of Lemma~2.2]{Breuning_Comp})
	\begin{align*}
	 	\norm{D^2u}_{C^0(B_r(0))}\leq 4\left(1+\norm{Du}_{C^0(B_r(0))}\right)^{\frac32} \norm{A},
	\end{align*}
	we obtain with appropriate choice of $ \alpha $
	\begin{align*}
		\norm{D^2u}_{C^0(B_r(0))} \leq 4\left(1+\alpha\right)^{\frac32} \norm{A} \leq 5 A_{\max}.
	\end{align*}
\end{proof}

\section*{Acknowledgments}
The authors thank Karsten Große-Brauckmann, Robert Kusner and Fabian Rupp for useful discussions and helpful comments.
The authors were partially funded by the Deutsche Forschungsgemeinschaft (DFG, German Research Foundation) under the project number MA 7559/1-2.
The second author would like to thank the Hausdorff Institute\footnote{funded by the Deutsche Forschungsgemeinschaft (DFG, German Research Foundation) under Germany's Excellence Strategy – EXC-2047/1 – 390685813} in Bonn for a welcoming and productive stay, where part of this work was completed. This work is part of the second author's PhD thesis.

\phantomsection\addcontentsline{toc}{section}{References}  
\bibliography{/home/seidel/Nextcloud/MyStuff/MyLit/000_bibs/total.bib}

\newcommand{\etalchar}[1]{$^{#1}$}
\begin{thebibliography}{HKMB24}

\bibitem[Ait10]{Aitchison_Holiverse}
Iain~R. Aitchison.
\newblock The `{H}oliverse': holistic eversion of the 2-sphere in
  $\mathbb{R}^3$, 2010.
\newblock \href{https://arxiv.org/abs/1008.0916}{arXiv:1008.0916}.

\bibitem[Ap{\' e}92]{Apery_SE}
Fran{\c c}ois Ap{\' e}ry.
\newblock An algebraic halfway model for the eversion of the sphere.
\newblock {\em Tohoku Math. J. (2)}, 44(1):103--150, 1992.
\newblock With an appendix by Bernard Morin.

\bibitem[BB19]{BB_SE}
Adam Bednorz and Witold Bednorz.
\newblock Analytic sphere eversion using ruled surfaces.
\newblock {\em Differential Geom. Appl.}, 64:59--79, 2019.

\bibitem[BK03]{BK_ExistenceWillmore}
Matthias Bauer and Ernst Kuwert.
\newblock Existence of minimizing {W}illmore surfaces of prescribed genus.
\newblock {\em Int. Math. Res. Not.}, (10):553--576, 2003.

\bibitem[BKP{\etalchar{+}}10]{BKPAL_Mesh}
Mario Botsch, Leif Kobbelt, Mark Pauly, Pierre Alliez, and Bruno Levy.
\newblock {\em Polygon Mesh Processing}.
\newblock A K Peters/CRC Press, October 2010.

\bibitem[Bla09]{Blatt}
Simon Blatt.
\newblock A singular example for the {W}illmore flow.
\newblock {\em Analysis (Munich)}, 29(4):407--430, 2009.

\bibitem[BR14]{BR_Quant}
Yann Bernard and Tristan Rivi\`ere.
\newblock Energy quantization for {W}illmore surfaces and applications.
\newblock {\em Ann. of Math. (2)}, 180(1):87--136, 2014.

\bibitem[Bre15]{Breuning_Comp}
Patrick Breuning.
\newblock Immersions with bounded second fundamental form.
\newblock {\em J. Geom. Anal.}, 25(2):1344--1386, 2015.

\bibitem[Bry84]{Bryant_Dual}
Robert~L. Bryant.
\newblock A duality theorem for {W}illmore surfaces.
\newblock {\em J. Differential Geom.}, 20(1):23--53, 1984.

\bibitem[CL14]{CL_BranchedWillmore}
Jingyi Chen and Yuxiang Li.
\newblock Bubble tree of branched conformal immersions and applications to the
  {W}illmore functional.
\newblock {\em Amer. J. Math.}, 136(4):1107--1154, 2014.

\bibitem[dC76]{DC}
Manfredo~P. do~Carmo.
\newblock {\em Differential geometry of curves and surfaces}.
\newblock Prentice-Hall, Inc., Englewood Cliffs, NJ, 1976.
\newblock Translated from the Portuguese.

\bibitem[DMRS25]{DMRS_WillmoreFlowToriFixedConformal}
Anna Dall'Acqua, Marius Müller, Fabian Rupp, and Manuel Schlierf.
\newblock {D}imension reduction for {W}illmore flows of tori: fixed conformal
  class and analysis of singularities, 2025.
\newblock \href{https://arxiv.org/abs/2502.12606}{arXiv:2502.12606}.

\bibitem[DMSS24]{DMSS_WillmoreFlowTOR}
Anna Dall'Acqua, Marius M\"uller, Reiner Sch\"atzle, and Adrian Spener.
\newblock The {W}illmore flow of tori of revolution.
\newblock {\em Anal. PDE}, 17(9):3079--3124, 2024.

\bibitem[FM80]{FM_SEShapiro}
George~K. Francis and Bernard Morin.
\newblock Arnold {S}hapiro's eversion of the sphere.
\newblock {\em Math. Intelligencer}, 2(4):200--203, 1979/80.

\bibitem[FSH98]{FSH_SE}
George Francis, John~M. Sullivan, and Chris Hartman.
\newblock Computing sphere eversions.
\newblock In {\em Mathematical visualization ({B}erlin, 1997)}, pages 237--255.
  Springer, Berlin, 1998.

\bibitem[FSK{\etalchar{+}}97]{FSKBHC_SE}
George Francis, John~M. Sullivan, Rob~B. Kusner, Ken~A. Brakke, Chris Hartman,
  and Glenn Chappell.
\newblock The minimax sphere eversion.
\newblock In {\em Visualization and mathematics ({B}erlin-{D}ahlem, 1995)},
  pages 3--20. Springer, Berlin, 1997.

\bibitem[Gor97]{Gor_Local}
Victor~V. Goryunov.
\newblock Local invariants of mappings of surfaces into three-space.
\newblock In {\em The {A}rnold-{G}elfand mathematical seminars}, pages
  223--255. Birkh\"auser Boston, Boston, MA, 1997.

\bibitem[Hel73]{Helfrich}
W.~Helfrich.
\newblock {E}lastic {P}roperties of {L}ipid {B}ilayers: {T}heory and {P}ossible
  {E}xperiments.
\newblock {\em Zeitschrift für Naturforschung C}, 28(11-12):693--703, 1973.

\bibitem[HKMB24]{HKMB}
Jonas Hirsch, Rob Kusner, and Elena M\"ader-Baumdicker.
\newblock Geometry of complete minimal surfaces at infinity and the {W}illmore
  index of their inversions.
\newblock {\em Calc. Var. Partial Differential Equations}, 63(8):Paper No. 190,
  30, 2024.

\bibitem[HMB23]{HMB}
Jonas Hirsch and Elena M\"ader-Baumdicker.
\newblock On the index of {W}illmore spheres.
\newblock {\em J. Differential Geom.}, 124(1):37--79, 2023.

\bibitem[Hug85]{Hughes_SE}
John~F. Hughes.
\newblock Another proof that every eversion of the sphere has a quadruple
  point.
\newblock {\em Amer. J. Math.}, 107(2):501--505, 1985.

\bibitem[HY17]{HY_SE}
Mikami Hirasawa and Minoru Yamamoto.
\newblock Sphere eversion from the viewpoint of generic homotopy.
\newblock {\em Topology Appl.}, 223:13--29, 2017.

\bibitem[KL12]{KL_W22conformal}
Ernst Kuwert and Yuxiang Li.
\newblock {$W^{2,2}$}-conformal immersions of a closed {R}iemann surface into
  {$\Bbb R^n$}.
\newblock {\em Comm. Anal. Geom.}, 20(2):313--340, 2012.

\bibitem[KLS10]{KLS_LargeGenus}
Ernst Kuwert, Yuxiang Li, and Reiner Sch\"atzle.
\newblock The large genus limit of the infimum of the {W}illmore energy.
\newblock {\em Amer. J. Math.}, 132(1):37--51, 2010.

\bibitem[Koe21]{Koerber_HawkingMass}
Thomas Koerber.
\newblock The area preserving {W}illmore flow and local maximizers of the
  {H}awking mass in asymptotically {S}chwarzschild manifolds.
\newblock {\em J. Geom. Anal.}, 31(4):3455--3497, 2021.

\bibitem[KS01]{KS_SmallInitial}
Ernst Kuwert and Reiner Sch\"{a}tzle.
\newblock The {W}illmore flow with small initial energy.
\newblock {\em J. Differential Geom.}, 57(3):409--441, 2001.

\bibitem[KS02]{KS_Gradient}
Ernst Kuwert and Reiner Sch\"atzle.
\newblock Gradient flow for the {W}illmore functional.
\newblock {\em Comm. Anal. Geom.}, 10(2):307--339, 2002.

\bibitem[KS04]{KS_Remov}
Ernst Kuwert and Reiner Sch\"atzle.
\newblock Removability of point singularities of {W}illmore surfaces.
\newblock {\em Ann. of Math. (2)}, 160(1):315--357, 2004.

\bibitem[KS12]{KS_Survey}
Ernst Kuwert and Reiner Sch\"{a}tzle.
\newblock The {W}illmore functional.
\newblock In {\em Topics in modern regularity theory}, volume~13 of {\em CRM
  Series}, pages 1--115. Ed. Norm., Pisa, 2012.

\bibitem[Kui61]{Kuiper_SE}
Nicolaas~H. Kuiper.
\newblock Convex immersions of closed surfaces in {$E\sp{3}$}. {N}onorientable
  closed surfaces in {$E\sp{3}$} with minimal total absolute {G}auss-curvature.
\newblock {\em Comment. Math. Helv.}, 35:85--92, 1961.

\bibitem[Kus87]{Kusner_Conformal}
Rob Kusner.
\newblock Conformal geometry and complete minimal surfaces.
\newblock {\em Bull. Amer. Math. Soc. (N.S.)}, 17(2):291--295, 1987.

\bibitem[Kus91]{K_Bubbles}
Rob Kusner.
\newblock Bubbles, conservation laws, and balanced diagrams.
\newblock In {\em Geometric analysis and computer graphics ({B}erkeley, {CA},
  1988)}, volume~17 of {\em Math. Sci. Res. Inst. Publ.}, pages 103--108.
  Springer, New York, 1991.

\bibitem[Lan85]{Langer_Comp}
Joel Langer.
\newblock A compactness theorem for surfaces with {$L_p$}-bounded second
  fundamental form.
\newblock {\em Math. Ann.}, 270(2):223--234, 1985.

\bibitem[LN15]{LN}
Tobias Lamm and Huy~The Nguyen.
\newblock Branched {W}illmore spheres.
\newblock {\em J. Reine Angew. Math.}, 701:169--194, 2015.

\bibitem[L{\'o}p92]{Lopez_CompleteMSTotCurv}
Francisco~J. L{\'o}pez.
\newblock The classification of complete minimal surfaces with total curvature
  greater than {$-12\pi$}.
\newblock {\em Trans. Amer. Math. Soc.}, 334(1):49--74, 1992.

\bibitem[LS84]{LS_Wtori}
Joel Langer and David Singer.
\newblock Curves in the hyperbolic plane and mean curvature of tori in
  {$3$}-space.
\newblock {\em Bull. London Math. Soc.}, 16(5):531--534, 1984.

\bibitem[LY82]{LY}
Peter Li and Shing~Tung Yau.
\newblock A new conformal invariant and its applications to the {W}illmore
  conjecture and the first eigenvalue of compact surfaces.
\newblock {\em Invent. Math.}, 69(2):269--291, 1982.

\bibitem[LZ24]{LeeZhao_Box}
Tang-Kai Lee and Xinrui Zhao.
\newblock {C}losed mean curvature flows with asymptotically conical
  singularities, 2024.
\newblock \href{https://arxiv.org/abs/2405.15577}{arXiv:2405.15577}.

\bibitem[Man18]{Mandel_WillmoreSOR}
Rainer Mandel.
\newblock Explicit formulas, symmetry and symmetry breaking for {W}illmore
  surfaces of revolution.
\newblock {\em Ann. Global Anal. Geom.}, 54(2):187--236, 2018.

\bibitem[Mar26]{Martino_BWS}
Dorian Martino.
\newblock Classification of branched {W}illmore spheres in codimension 1.
\newblock {\em Proc. Amer. Math. Soc.}, 154(4):1541--1556, 2026.

\bibitem[MB81]{MB_SE}
Nelson Max and Tom Banchoff.
\newblock Every sphere eversion has a quadruple point.
\newblock In {\em Contributions to analysis and geometry ({B}altimore, {M}d.,
  1980)}, pages 191--209. Johns Hopkins Univ. Press, Baltimore, MD, 1981.

\bibitem[Mic21]{M_BWSIndex}
Alexis Michelat.
\newblock On the {M}orse index of branched {W}illmore spheres in 3-space.
\newblock {\em Calc. Var. Partial Differential Equations}, 60(4):Paper No. 126,
  97, 2021.

\bibitem[Miu23]{Miura_LiYau}
Tatsuya Miura.
\newblock Li-{Y}au type inequality for curves in any codimension.
\newblock {\em Calc. Var. Partial Differential Equations}, 62(8):Paper No. 216,
  28, 2023.

\bibitem[MN14a]{MN_WillmoreConjecture}
Fernando~C. Marques and Andr\'e Neves.
\newblock Min-max theory and the {W}illmore conjecture.
\newblock {\em Ann. of Math. (2)}, 179(2):683--782, 2014.

\bibitem[MN14b]{MN_WillmoreConjectureSurvey}
Fernando~C. Marques and Andr\'e Neves.
\newblock The {W}illmore conjecture.
\newblock {\em Jahresber. Dtsch. Math.-Ver.}, 116(4):201--222, 2014.

\bibitem[MN18]{MN_Conformal}
Andrea Mondino and Huy~T. Nguyen.
\newblock Global conformal invariants of submanifolds.
\newblock {\em Ann. Inst. Fourier (Grenoble)}, 68(6):2663--2695, 2018.

\bibitem[MP78a]{Morin_SERet}
Bernard Morin and Jean-Pierre Petit.
\newblock Le retournement de la sphere.
\newblock {\em C. R. Acad. Sci. Paris S\'er. A-B}, 287(11):A791--A794, 1978.

\bibitem[MP78b]{Morin_SEProb}
Bernard Morin and Jean-Pierre Petit.
\newblock Probl\'ematique du retournement de la sph\`ere.
\newblock {\em C. R. Acad. Sci. Paris S\'er. A-B}, 287(10):A767--A770, 1978.

\bibitem[MR22]{MR_BWS_S3S4}
Alexis Michelat and Tristan Rivi\`ere.
\newblock The classification of branched {W}illmore spheres in the 3-sphere and
  the 4-sphere.
\newblock {\em Ann. Sci. \'Ec. Norm. Sup\'er. (4)}, 55(5):1199--1288, 2022.

\bibitem[MR23]{MR_LY1dim}
Marius M\"uller and Fabian Rupp.
\newblock A {L}i-{Y}au inequality for the 1-dimensional {W}illmore energy.
\newblock {\em Adv. Calc. Var.}, 16(2):337--362, 2023.

\bibitem[MRS26]{MRS_closedGeodesic}
Marius M\"uller, Fabian Rupp, and Christian Scharrer.
\newblock Short closed geodesics and the {W}illmore energy.
\newblock {\em J. Differential Geom.}, 132(3):429--459, 2026.

\bibitem[MS02]{MS_Num}
Uwe~F. Mayer and Gieri Simonett.
\newblock A numerical scheme for axisymmetric solutions of curvature-driven
  free boundary problems, with applications to the {W}illmore flow.
\newblock {\em Interfaces Free Bound.}, 4(1):89--109, 2002.

\bibitem[Mus20]{M_Germain}
Dora Musielak.
\newblock {\em Elasticity Theory After Germain}, pages 103--114.
\newblock Springer International Publishing, Cham, 2020.

\bibitem[P{\'e}r98]{Perez_RelMS}
Joaqu{\'i}n P{\'e}rez.
\newblock Riemann bilinear relations on minimal surfaces.
\newblock {\em Math. Ann.}, 310(2):307--332, 1998.

\bibitem[Phi66]{Phillips_SE}
Anthony Phillips.
\newblock Turning a surface inside out.
\newblock {\em Sci. Amer.}, 214(5):112, 1966.

\bibitem[PR24]{PR_Param}
Francesco Palmurella and Tristan Rivi\`ere.
\newblock The parametric {W}illmore flow.
\newblock {\em J. Reine Angew. Math.}, 811:1--91, 2024.

\bibitem[PS87]{PS_WillmoreSurfs}
U.~Pinkall and I.~Sterling.
\newblock Willmore surfaces.
\newblock {\em Math. Intelligencer}, 9(2):38--43, 1987.

\bibitem[Riv08]{R_AnalysisWillmore}
Tristan Rivi\`ere.
\newblock Analysis aspects of {W}illmore surfaces.
\newblock {\em Invent. Math.}, 174(1):1--45, 2008.

\bibitem[Riv12]{R_WillmoreNotes}
Tristan Rivière.
\newblock {C}onformally {I}nvariant {V}ariational {P}roblems, 2012.
\newblock Lecture notes,
  \href{https://arxiv.org/abs/1206.2116}{arXiv:1206.2116}.

\bibitem[Riv21]{R_SE}
Tristan Rivi\`ere.
\newblock Willmore minmax surfaces and the cost of the sphere eversion.
\newblock {\em J. Eur. Math. Soc. (JEMS)}, 23(2):349--423, 2021.

\bibitem[Sch25]{Schlierf_LiYauBoundary}
Manuel Schlierf.
\newblock Global existence for the {W}illmore flow with boundary via {S}imon's
  {L}i-{Y}au inequality.
\newblock {\em Adv. Calc. Var.}, 18(2):555--575, 2025.

\bibitem[Sei25]{S_TriplePoints}
Jona Seidel.
\newblock An {I}nvariant for {T}riple-{P}oint-{F}ree {I}mmersed {S}pheres,
  2025.
\newblock \href{https://arxiv.org/abs/2506.21130}{arXiv:2506.21130}.

\bibitem[Sim93]{Simon_W}
Leon Simon.
\newblock Existence of surfaces minimizing the {W}illmore functional.
\newblock {\em Comm. Anal. Geom.}, 1(2):281--326, 1993.

\bibitem[Sim01]{Simonett_WillmoreFlow}
Gieri Simonett.
\newblock The {W}illmore flow near spheres.
\newblock {\em Differential Integral Equations}, 14(8):1005--1014, 2001.

\bibitem[Sma58]{Smale_SE}
Stephen Smale.
\newblock A classification of immersions of the two-sphere.
\newblock {\em Trans. Amer. Math. Soc.}, 90:281--290, 1958.

\bibitem[Sma59]{Smale_DiffS2}
Stephen Smale.
\newblock Diffeomorphisms of the {$2$}-sphere.
\newblock {\em Proc. Amer. Math. Soc.}, 10:621--626, 1959.

\bibitem[Sto26]{Stolarski_Box}
Maxwell Stolarski.
\newblock Closed {R}icci flows with singularities modeled on asymptotically
  conical shrinkers.
\newblock {\em Geom. Topol.}, 30(4):1277--1370, 2026.

\end{thebibliography}
\end{document}